\documentclass{amsproc}
\usepackage{epsf,latexsym,graphicx}
\usepackage{amssymb}

\newtheorem{theorem}{Theorem}[section]
\newtheorem{lemma}[theorem]{Lemma}

\theoremstyle{definition}
\newtheorem{definition}[theorem]{Definition}
\newtheorem{example}[theorem]{Example}

\theoremstyle{remark}
\newtheorem{remark}[theorem]{Remark}

\renewcommand{\theequation}{\thesection.\arabic{theorem}}

\newtheorem{corollary}[theorem]{Corollary}
\newtheorem{examples}[theorem]{Examples}
\newtheorem{proposition}[theorem]{Proposition}

\newtheorem{question}[theorem]{Question}

\newtheorem{paragraf}[theorem]{}
\newtheorem{conjecture}[theorem]{Conjecture}
\newtheorem{notation}[theorem]{Notation}
\newtheorem{problem}[theorem]{Problem}
\newtheorem{thevarthm}[theorem]{\varthmname}

\newenvironment{varthm*}[1]{\trivlist\item[]{\bf #1.}\it}{\endtrivlist}

\renewcommand\emptyset{\varnothing}  
\renewcommand\ge{\geqslant}  
\renewcommand\le{\leqslant}  
\renewcommand\geq{\geqslant}  
\renewcommand\leq{\leqslant}  
\renewcommand\epsilon{\varepsilon}
\renewcommand\phi{\varphi}

\renewcommand\tilde{\widetilde}

\newcommand\be{\begin{eqnarray*}}
\newcommand\ee{\end{eqnarray*}}
\renewenvironment{cases}{\left\{\begin{array}{@{}l@{\quad}l}}{\end{array}\right.}
\newcommand\eqnref[1]{(\ref{#1})}

\newcommand\eps{\varepsilon}
\renewcommand\P{\mathbb P}
\newcommand\N{\mathbb N}
\newcommand\Q{\mathbb Q}
\newcommand\R{\mathbb R}
\newcommand\Z{\mathbb Z}
\newcommand\C{\mathbb C}

\newcommand\lra{\longrightarrow}
\newcommand\newop[2]{\def#1{\mathop{\rm #2}\nolimits}}
\newop\upper{upper}
\newop\SB{SB}
\newcommand\calo{{\mathcal O}}
\newcommand\cali{{\mathcal I}}
\newcommand\tensor{\otimes}

\newcommand\mylabel[1]{\label{#1}}

\newcommand\mult{{\rm mult}}

\newcommand\epsmov{\eps_{\rm mov}}

\def\ben{\begin{eqnarray}}
\def\een{\end{eqnarray}}

\def\binom#1#2{{#1\choose #2}}
\def\eatit#1{}
\newop\Ass{Ass}
\newop\codim{codim}
\newop\reg{reg}

\begin{document}

\title{A primer on Seshadri constants}

\author[Bauer]{Thomas Bauer}
\address{Fachbereich Mathematik und Informatik,
Philipps-Universit\"at Marburg, Hans-Meerwein-Stra\ss{}e, D-35032
Marburg, Germany}
\email{tbauer@Mathematik.Uni-Marburg.de }
\thanks{This work has been partially supported by
   the SFB/TR 45 ``Periods, moduli spaces and arithmetic of algebraic
   varieties'', in particular through supporting the workshop in Essen
   out of which this paper outgrew.}

\author[Di Rocco]{Sandra Di Rocco}
\address{Department of Mathematics, KTH, 100 44 Stockholm, Sweden}
\email{dirocco@math.kth.se}
\thanks{The second named author was partially supported by Vetenskapsr{\aa}det's grant
   NT:2006-3539.}

\author[Harbourne]{Brian Harbourne}
\address{Department of Mathematics, University of Nebraska--Lincoln, Lincoln, NE 68588-0130 USA}
\email{bharbourne1@math.unl.edu}

\author[Kapustka]{Micha\l \ Kapustka}
\address{Institute of Mathematics UJ, Reymonta 4, 30-059 Krak\'ow, Poland}
\email{michal.kapustka@im.uj.edu.pl}
\thanks{The fourth and the last two named authors were partially supported by
   a MNiSW grant N~N201 388834.}

\author[Knutsen]{Andreas Knutsen}
\address{Department of Mathematics, University of Bergen, Johannes Brunsgate 12, 5008 Bergen, Norway}
\email{Andreas.Knutsen@math.uib.no}

\author[Syzdek]{Wioletta Syzdek}
\address{Instytut Matematyki AP, ul. Podchor\c a\.zych 2, PL-30-084 Krak\'ow, Poland}
\curraddr{Mathematisches Institut, Universit\"at Duisburg-Essen, 45117 Essen, Germany}
\email{wioletta.syzdek@uni-due.de}

\author[Szemberg]{Tomasz Szemberg}
\address{Instytut Matematyki AP, ul. Podchor\c a\.zych 2, PL-30-084 Krak\'ow, Poland}
\curraddr{Instytut Matematyczny PAN, ul. \'Sniadeckich 8, PL-00-956 Warszawa, Poland}
\email{tomasz.szemberg@uni-due.de}

\subjclass{Primary 14C20; Secondary 14E25, 14J26, 14M25, 13A10}
\date{September 18, 2008}

\dedicatory{This paper is dedicated to Andrew J. Sommese.}

\keywords{Seshadri constants, linear series, symbolic powers}

\begin{abstract}
   Seshadri constants express the so called local positivity of a line
   bundle on a projective variety. They were introduced in
   \cite{Dem92} by Demailly. The original hope of using them towards
   a proof of the Fujita conjecture was too optimistic, but it soon
   became clear that they are interesting invariants quite in their own
   right. Lazarsfeld's book \cite{PAG} contains a whole chapter
   devoted to local positivity and serves as a very enjoyable
   introduction to Seshadri constants. Since this book has appeared,
   the subject witnessed quite a bit of development. It is the aim
   of these notes to give an account of recent progress as well as to
   discuss many open questions and provide some examples. The idea
   of writing these notes
   occurred during the workshop on Seshadri constants held in Essen
   12-15 February 2008.
\end{abstract}

\maketitle

\tableofcontents
\section{Definitions}

   We begin by recalling the Seshadri criterion for ampleness
   \cite[Theorem 1.7]{Har70},
   as this is where the whole story begins.
\begin{theorem}[Seshadri criterion]\mylabel{Seshcrit}
   Let $X$ be a smooth projective variety and $L$ be a line bundle
   on $X$. Then $L$ is ample if and only if there exists a positive number
   $\eps$ such that for all points $x$ on $X$ and all (irreducible) curves $C$
   passing through $x$ one has
   $$L\cdot C\geq \eps\cdot\mult_xC.$$
\end{theorem}
\begin{remark}[Insufficiency of positive intersections with curves]\rm
   It is not enough to assume merely that the intersection of $L$
   with every curve is positive. In other words it is not enough to
   assume that $L$ restricts to an ample line bundle on every curve
   $C\subset X$. Counterexamples were constructed by Mumford and
   Ramanujam \cite[Examples 10.6 and 10.8]{Har70}.
\end{remark}
   It is natural to ask for optimal numbers $\eps$ in Theorem
   \ref{Seshcrit}. This leads to the following definition due to
   Demailly \cite{Dem92}.
\begin{definition}[Seshadri constant at a point]\rm\mylabel{defDem}
   Let $X$ be a smooth projective variety and $L$ a nef line bundle
   on $X$. For a fixed point $x\in X$ the real number
   $$\eps(X,L;x):=\inf\frac{L\cdot C}{\mult_xC}$$
   is the \emph{Seshadri constant of $L$ at $x$} (the
   infimum being taken over all irreducible curves $C$ passing through
   $x$).
\end{definition}
\begin{definition}[Seshadri curve]\mylabel{sescur}\rm
   We say that a curve $C$ is a \emph{Seshadri curve of $L$ at $x$}
   if $C$ computes $\eps(X,L;x)$, i.e., if
   $$\eps(X,L;x)=\frac{L\cdot C}{\mult_xC}.$$
\end{definition}
   It is not known if Seshadri curves exist in general.

   Definition \ref{defDem} extends naturally so that we can define
   Seshadri constants for an arbitrary subscheme $Z\subset
   X$. To this end let $f:Y\lra X$ be the blowup of $X$ along $Z$
   with the exceptional divisor $E$.
\begin{definition}[Seshadri constant at a subscheme]\rm\mylabel{defsubscheme}
   The \emph{Seshadri constant} of $L$ at $Z$ is the real number
\addtocounter{theorem}{1}
   \ben\label{defnef}
   \eps(X,L;Z):=\sup\left\{\lambda\,:\, f^*L-\lambda E
   \mbox{ is ample on }Y\right\}.
   \een
\end{definition}
\begin{remark}\rm
   If $Z$ is a point, then both definitions agree. The argument is
   given in \cite[Proposition 5.1.5]{PAG}.
\end{remark}
\begin{remark}[Relation to the $s$-invariant]\rm
   Note that $\eps(X,L;Z)$ is the reciprocal of the $s$-invariant $s_L(\cali_Z)$ of
   the ideal sheaf $\cali_Z$ of $Z$ with respect to $L$ as defined
   in \cite[Definition 5.4.1]{PAG}
\end{remark}
\begin{definition}[Multi-point Seshadri constant]\mylabel{defmultisesh}
   If $Z$ is a reduced subscheme
   supported at $r$ distinct points
   $x_1,\dots,x_r$ of $X$, then the number
   $\eps(X,L;x_1,\dots,x_r)$ is called the \emph{multi-point
   Seshadri constant} of $L$ at the $r$-tuple of points $x_1,\dots,x_r$.
\end{definition}
   There is yet another variant of Definition \ref{defDem} which
   instead of curves takes into account higher dimensional
   subvarieties of $X$ passing through a given point $x\in X$.
\begin{definition}[Seshadri constants via higher dimensional subvarieties]\rm\mylabel{defhighdim}
   Let $X$ be a smooth projective variety, $L$ a nef line bundle on
   $X$ and $x\in X$ a point. The real number
   $$\eps_d(X,L;x):=\inf\left(\frac{L^d\cdot V}{\mult_xV}\right)^{\frac1d}$$
   is the $d$-\emph{dimensional Seshadri constant} of $L$ at $x$
   (the infimum being taken over all subvarieties $V\subset X$ of
   dimension $d$ such that $x\in V$).
\end{definition}
\begin{remark}\rm
   Note that the above definition agrees for $d=1$ with Definition
   \ref{defDem}, so that $\eps(X,L;x)=\eps_1(X,L;x)$.
\end{remark}
   In the above definitions we suppress the variety $X$ if it is
   clear from the context where the Seshadri constant is computed,
   i.e., we write $\eps(L;x)=\eps(X,L;x)$ etc.

   There are another three interesting numbers which can be defined taking
   infimums over various spaces of parameters.
\begin{definition}[Seshadri constants of a line bundle, a point and a varie\-ty]\rm\mylabel{infs}
\mbox{}
\begin{itemize}
\item[(a)]   The number
   $$\eps(X,L):=\inf_{x\in X}\eps(X,L;x)$$
   is the \emph{Seshadri constant of the line bundle} $L$.
\item[(b)]
   The number
   $$\eps(X;x):=\inf_{L \mbox{ \scriptsize ample}}\eps(X,L;x)$$
   is the \emph{Seshadri constant of the point} $x\in X$.
\item[(c)]
   The number
   $$\eps(X):=\inf_{L \mbox{ \scriptsize ample}}\eps(X,L)=\inf_{x\in X}\eps(X;x)$$
   is the \emph{Seshadri constant of the variety} $X$.
\end{itemize}
\end{definition}
\begin{remark}[Reformulation of Seshadri criterion]\rm
   Theorem \ref{Seshcrit} asserts now simply that a line bundle $L$ is
   ample if and only if its Seshadri constant is positive: $\eps(X,L)>0$.
\end{remark}
   So far we defined Seshadri constants for ample or at least nef
   line bundles. Recently Ein, Lazarsfeld, Mustata, Nakamaye and
   Popa \cite{RV} found a meaningful way to extend the notion of Seshadri
   constants to big line bundles.

   To begin with, we recall the notion of augmented base locus. For this purpose it is
   convenient to pass to $\Q$-divisors.
\begin{definition}[Augmented base locus]\rm\mylabel{augmentedbs}
   Let $D$ be a $\Q$-divisor. The \emph{augmented base locus} of
   $D$ is
   $$B_+(D):=\bigcap_A\SB(D-A),$$
   where the intersection is taken over all sufficiently small ample
   $\Q$-divisors $A$ and $\SB(D-A)$ is the stable base locus of
   $D-A$, i.e., the common base locus of all linear series $|m(D-A)|$
   for all sufficiently divisible $m$. (In fact $B_+(D)=\SB(D-A)$
   for any sufficiently small ample $A$.)
\end{definition}
\begin{remark}[Numerical nature of augmented base loci]\rm
   Contrary to the stable base loci, the augmented base loci depend
   only on the numerical class of $D$ \cite[Proposition 1.4]{ELMNP06}.
\end{remark}
   Intuitively, the augmented base locus of a line bundle $L$ is the
   locus where $L$ has no local positivity. This is reflected by
   the following definition.
\begin{definition}[Moving Seshadri constant]\label{moving_sc}\rm
   Let $X$ be a smooth projective variety and $L=\calo_{X}(D)$ a line bundle on
   $X$. The real number
   $$\epsmov(L;x):=\left\{
     \begin{array}{cl}
        \sup_{f^*D=A+E}\eps(A;x) & \mbox{ if } x \mbox{ is not in }
        B_+(L),\\
        0 & \mbox{ otherwise}.
     \end{array}
     \right.$$
   is the \emph{moving Seshadri constant} of $L$ at $x$. The
   supremum in the definition is taken over all projective
   morphisms $f:X'\mapsto X$, with $X'$ smooth, which are
   isomorphisms over a neighborhood of $x$ and all decompositions
   $f^*(D)=A+E$ such that $E$ is an effective
   $\Q$-divisors and $A=f^*(D)-E$ is ample.
\end{definition}
   Note that if $L$ is not big, then $\epsmov(L;x)=0$ for every
   point $x\in X$, so the moving Seshadri constants are meaningful
   for big divisors only.
\begin{remark}[Consistency of definitions]\rm
   If $L$ is nef, then the above definition agrees with Definition
   \ref{defDem}. One can also state the other definitions of this section
   in the moving context. This is left to the reader.
\end{remark}
   We conclude with yet another remark relating moving Seshadri constants
   to Zariski decompositions on surfaces.
   The definition of the Zariski decomposition is provided by the following
   theorem, see \cite{Zar62} and \cite{Bau08}.
\begin{theorem}[Zariski decomposition]\mylabel{zardec}
   Let $D$ be an effective $\Q$-divisor on a~smooth projective surface $X$.
   Then there are uniquely determined effective (possibly zero) $\Q$-divisors
   $P$ and $N$ with $D\; =\; P\; +\; N$ such that:
   \begin{itemize}
      \item[(i)] $P$ is nef ;
      \item[(ii)] $N$ is zero or has negative definite intersection matrix ;
      \item[(iii)] $P\cdot C = 0$ for every irreducible component $C$ of $N$.
   \end{itemize}
\end{theorem}
\begin{remark}[Moving Seshadri constants and Zariski decompositions]\rm
   Let $L=\calo_X(D)$ be a big line bundle on a smooth projective surface $X$
   and let $D=P+N$ be the Zariski decomposition of $D$, then
   $$\eps_{mov}(L;x)=\eps(P;x)\;.$$
\end{remark}
\begin{proof}
   First of all recall that one has
\addtocounter{theorem}{1}
   \begin{equation}\label{zd_sections}
   H^0(mL)=H^0(mP)
   \end{equation}
   for all $m$ sufficiently divisible.
   Then \eqnref{nef_jets} relates $\eps(P;x)$ to the number
   of jets generated asymptotically by $P$ at $x$. The same
   relation holds for moving Seshadri constants by \cite[Proposition 6.6]{RV}.
   Taking \eqnref{zd_sections} into account we have
   $$\eps_{mov}(L;x)=\sup_m \frac{s(mL,x)}{m}=\sup_m \frac{s(mP,x)}{m}=\eps(P;x).$$
\end{proof}

\begingroup

\section{Basic properties}

\setcounter{theorem}{0}
\renewcommand{\thetheorem}{\thesubsection.\arabic{theorem}}
\renewcommand{\theequation}{\thesubsection.\arabic{theorem}}

\subsection{Upper bounds and submaximal curves}

   Since Seshadri constants are in particular defined by a nefness condition, it is
   easy to come up with an upper bound using Kleiman's
   criterion \cite[Theorem 1.4.9]{PAG}. For
   $0$-dimensional reduced subschemes we have the following result.
\begin{proposition}[Upper bounds]\mylabel{upperbound}
   Let $X$ be a smooth projective variety of dimension $n$ and $L$ a
   nef line bundle on $X$. Let $x_1,\dots,x_r$ be $r$ distinct
   points on $X$, then
   $$\eps(X,L;x_1,\dots,x_r)\leq\sqrt[n]{\frac{L^n}{r}}.$$
   In particular for a single point $x$ we always have
   $$\eps(X,L;x)\leq\sqrt[n]{L^n}.$$
\end{proposition}
\begin{proof}
   Let $f:Y\lra X$ be the blowup $x_1,\dots,x_r$. Then the
   exceptional divisor $E=E_1+\dots +E_r$ is the sum of disjoint
   exceptional divisors over each of the points. By \eqnref{defnef}
   we must have $(f^*L-\eps(X,L;x_1,\dots,x_r)E)^n\geq 0$, and the claim
   follows.
\end{proof}
   The above proposition leads in a natural manner to the following
   definition.
\begin{definition}[Submaximal Seshadri constants]\rm
   We say that the Seshadri constant $\eps(X,L;x)$ is
   \emph{submaximal} if the strict inequality holds
   $$\eps(X,L;x)<\sqrt[n]{L^n}\;.$$
\end{definition}
   The above definition is paralleled by the following one.
\begin{definition}[Submaximal curves]\label{submaxcur}\rm
   Let $X$ be a smooth projective surface and $L$ an ample line
   bundle on $X$. We say that $C\subset X$ is a \emph{submaximal
   curve} (at $x\in X$ with respect to $L$) if
   $$\frac{L\cdot C}{\mult_x C}<\sqrt{L^2}\;.$$
   If only the weak inequality holds for $C$, then we call $C$ a
   \emph{weakly-submaximal} curve.\\
\end{definition}
\begin{remark}\label{sub-sur-rat}\rm
   For surfaces submaximal Seshadri constants are always computed
   by Seshadri curves, see \cite[Proposition 1.1]{BauSze08}. In particular they are rational numbers.
\end{remark}
   In general we have the following restriction on possible values
   of Seshadri constants \cite[Prop. 4]{Ste98}, which is
   a direct consequence of the Nakai-Moishezon criterion for $\R$-divisors
   \cite{Cam-Pet90}.
\begin{theorem}[Submaximal Seshadri constants are roots] \label{thm:submax}
   Let $X$ be an $n$-dimensional smooth projective variety,
   $L$ an ample line bundle on $X$ and $x$ a point of $X$.

   If $\varepsilon(L,x)$ is submaximal, that is,
   $\varepsilon(L,x) < \sqrt[n]{L^n}$, then it is a $d$-th root of a~rational number,
   for some $d$ with $1 \leq d \leq n-1$.
\end{theorem}
   In particular, it might happen that a Seshadri constant is computed by a higher dimensional subscheme.
   It is interesting to note that $d$-dimensional Seshadri
   constants are partially ordered \cite[Proposition 5.1.9]{PAG}.
\begin{proposition}[Relation between $d$-dimensional Seshadri
constants]\label{SCrltn}
   For a line bundle $L$ on a smooth projective variety $X$ of dimension
   $n$, a point $x\in X$ and an integer $d$ with $1\leq d\leq n$ we have
   $$\eps(L;x)\leq \eps_d(L;x).$$
\end{proposition}
   Note that for $d$ we just recover the bound from Proposition
   \ref{upperbound} with $r=1$.

   Recently Ross and Ro\'e \cite[Remark 1.3]{RosRoe}
   have raised the interesting question if
   $$\eps_{d_1}(L;x)\leq \eps_{d_2}(L;x)$$
   for all $d_1\leq d_2$ (and the analogous version in the multi-point setting).

   \setcounter{theorem}{0}
\renewcommand{\thetheorem}{\thesubsection.\arabic{theorem}}
\renewcommand{\theequation}{\thesubsection.\arabic{theorem}}

\subsection{Lower bounds}\label{low-bou}
   Now we turn our attention to lower bounds.
   Extrapolating on Definition \ref{infs}, one could hope that yet
   another infimum can be taken: For a positive integer $n$ define
   $$\eps(n):=\inf\eps(X),$$
   where the infimum is taken this time over all smooth projective
   varieties of dimension $n$. However the numbers $\eps(n)$
   always equal zero. Miranda (see \cite[Example 5.2.1]{PAG})
   constructed a sequence of examples
   of smooth surfaces $X_n$, ample line bundles $L_n$ on $X_n$
   and points $x_n\in X_n$ such that
   $$\lim_{n\rightarrow\infty}\eps(X_n,L_n;x_n)=0.$$
   Miranda's construction was generalized to arbitrary dimension by
   Viehweg (see \cite[Example 5.2.2]{PAG}). In these examples only
   rational varieties were used but it was quickly realized in
   \cite[Proposition 3.3]{Bau99} that the same phenomenon happens
   on suitable blow ups of arbitrary varieties. Note that in the above sequence it is
   necessary to change the underlying variety all the time. It is
   natural to ask if one could realize the sequence $(L_n,x_n)$ as
   above on a single variety $X$, i.e., to raise the following problems.
\begin{question}[Existence of a lower bound on a fixed variety]\mylabel{lowerx}$\;$
   \begin{itemize}
      \item[(a)] Can it happen that $\eps(X)=0$?
      \item[(b)] If not, is it possible to compute a lower bound in
      terms of geometric invariants of $X$?
   \end{itemize}
\end{question}
   This question was asked already in the pioneering paper of
   Demailly \cite[Question 6.9]{Dem92}. Up to now, we don't know.
   However there is one obvious instance in which there is a
   negative answer to Question \ref{lowerx}(a), namely if the Picard
   number $\rho(X)$ is equal to $1$. In case of surfaces there is
   also a sharp answer to Question \ref{lowerx}(b). We come back to this in Theorem \ref{eff-bound-sur}.

   Another class of varieties, where answers to Question
   \ref{lowerx} are known,
   is constituted by abelian varieties.
   First of all, since on an abelian variety one can translate
   divisors around without changing their numerical class, it is
   clear that one has the lower bound
\addtocounter{theorem}{1}
   \ben\label{ab-bound}
      \eps(X,L)\geq 1
   \een
   for any ample line bundle $L$ on an abelian variety $X$.
   A beautiful result of Nakamaye \cite{Nak96} gives precise
   characterization of when there is equality in
   \eqnref{ab-bound}.
\begin{theorem}[Seshadri constants on abelian varieties]\mylabel{ab-thm}
   Let $(X,L)$ be a polarized abelian variety. Then
   $\eps(L)= 1$
   if and only if $X$ splits off an elliptic
   curve and the polarization splits as well, i.e.,
   $$X=X'\times E \mbox{ and } L=\pi_1^*(L')\otimes\pi_2^*(L_E),$$
   where $E$ is an elliptic curve, $X'$ an abelian variety, $L_E,
   L'$ are ample line bundles on $E$ and $X'$ respectively and
   $\pi_i$ are projections in the product.
\end{theorem}
   Furthermore, a lower bound for the Seshadri constant
   $\eps(X)$ of a variety $X$ can always be given, provided one has good control over
   base point freeness or very ampleness of ample line bundles on
   $X$. Specifically we have the following fact
   \cite[Example 5.1.18]{PAG}.
\begin{proposition}[Lower bound for spanned line bundles]\label{lowspan}
   Let $L$ be an ample and spanned line bundle on a smooth
   projective variety $X$, then
   $$\eps(X,L;x)\geq 1$$
   for all points $x\in X$.
\end{proposition}
   This proposition generalizes easily to the case when $L$
   generates $s$-jets at a~point, i.e., when the evaluation mapping
   $$H^0(X,L)\lra H^0(X,L\otimes\calo_X/\cali_x^{s+1})$$
   is surjective. (Here $\cali_x$ denotes the ideal sheaf of a point $x\in X$.)
\begin{proposition}[Lower bound under generation of higher jets]\label{lowjet}
   Let $L$ be an ample line bundle generating $s$-jets (for $s\geq 1$)
   at a point
   $x$ of a smooth projective variety $X$. Then
   $$\eps(X,L;x)\geq s.$$
   In particular, if $L$ is very ample, then $\eps(L;x)\geq 1$ for
   all points $x\in X$.
\end{proposition}

   The above proposition is a special case of the following
   characterization of Seshadri constants via generation of jets.
   Denote for $k\ge 1$
   by $s(kL,x)$ the maximal integer $s$
   such that the linear series $|kL|$ generates $s$-jets at $x$.
   Then one has for $L$ nef,
\addtocounter{theorem}{1}
   \begin{equation}\label{nef_jets}
      \eps(L;x)=\sup \frac{s(kL,x)}k
   \end{equation}
   (see \cite[6.3]{Dem92}). If $L$ is ample, then the supremum is in fact a limit:
   $$
      \eps(L;x)=\lim_{k\to\infty} \frac{s(kL,x)}k.
   $$

   Whereas Question \ref{lowerx} has remained unanswered for several
   years, one can raise a seemingly easier problem concerning the Seshadri
   constant at a fixed point $x\in X$.
\begin{question}[Existence of a lower bound at a fixed point]
   Can it happen that $\eps(X,x)=0$?
\end{question}
   As of this writing we don't know the answer, even for surfaces.

\begin{paragraf}[Seshadri function]\rm\label{sesfun}
   Definition \ref{moving_sc} generalizes easily to $\R$-divisors
   and it is clear that it depends only on the numerical class
   of $D$. So, we can consider Seshadri constants for
   elements of the N\'eron-Severi space $N^1(X)_{\R}$. It is then
   reasonable to ask about regularity properties of the mapping
   $$\epsmov(X,\cdot;\cdot): N^1(X)_{\R}\times X\ni (L,x)\mapsto
   \epsmov(X,L;x)\in\R.$$
   It turns out that this mapping is continuous with respect to the
   first variable \cite[Theorem 6.2]{RV} and lower semi-continuous
   with respect to the second variable (in the topology for which closed
   sets are countable unions of Zariski closed sets)
   \cite[Example 5.1.11]{PAG}.
\end{paragraf}

\endgroup


\section{Projective spaces}\label{rec-est}
   The case of $\mathbb{P}^2$ polarized by $\mathcal{O}_{\mathbb{P}^2}
   (1)$ attracts most of the attention devoted to multiple point
   Seshadri constants. Thanks to a good interpretation in terms of
   polynomials the problem of estimating Seshadri constants is well
   tractable by computer calculations. This, together with the
   motivation to handle the still open Nagata conjecture, has caused
   a lot of effort to find lower estimates
   for general multiple point Seshadri constants on $\mathbb{P}^2$
   which are as precise as possible. In
   many cases analogous methods can also be applied in higher
   dimensions.

   For now the best estimates are obtained by M. Dumnicki using a
   combination of two methods contained in \cite{HR03} and
   \cite{Dum07}. Both methods appear in a different context and
   complement each other. The first gives us a relatively small family of
   all possible divisor classes that might contain curves which compute
   the Seshadri constants, whereas the second enables us to check if a
   linear system is empty.

   We need the following generalization of Definition
   \ref{submaxcur}.
\begin{definition}[Multi-point weakly-submaximal curve]\mylabel{submaxcurve}\rm
   Let $X$ be a smooth projective variety of dimension
   $n$ and $L$ an ample line bundle on $X$. Let $x_1,\dots,x_r \in X$
   be $r$ arbitrary distinct points. We say that a curve $C$ is
   \emph{weakly-submaximal} for $L$ with respect to these points if
   $$\frac{L\cdot C}{\sum\mult_{x_i}(C)}\leq \sqrt[n]{\frac{L^n}{r}}\;.$$
\end{definition}
   In the view of Proposition \ref{upperbound} weakly-submaximal curves are
   important because they contribute substantially to the infimum
   in Definition \ref{defDem}. It is not known in general if weakly-submaximal
   curves exist. In any case if there are no weakly-submaximal curves
   for $L$ with respect to the given points, then the Seshadri constant
   computed in these points equals $\sqrt[n]{L^n/r}$.

   The following theorem \cite{HR03} restricts the set of candidates
   for divisor classes of weakly-submaximal curves in $\P^2$ under the assumption
   that the points $x_1,\dots,x_r$ are in general position.
\begin{theorem}[Restrictions on weakly-submaximal curves]
   Let $X$ be obtained by blowing up
   $r\geq 10$ general points $p_1,\dots,p_r\in \P^2$
   and let $L$ be the pull-back of the hyperplane bundle on $\P^2$. If $H$ is
   the class of a proper transform to $X$ of a~weakly-submaximal curve, then
   there exist integers $t$,$m>0$ and $k$ such that:
   \begin{itemize}
      \item[(a)] $H=tL-m(E_1+ \dots +E_r)-kE_i\;$;
      \item[(b)] $-m<k$ and $k^2< \frac{r}{r-1} \min\{m,m+k\}\;$;
      \item[(c)] $\begin{cases}\rm
        m^2r+2mk+\max \{k^2-m,0)\}\leq t^2\leq m^2r+2mk+\frac{k^2}{r},  \text{ when } k> 0;
        \\ m^2r-m\leq t^2 < m^2r,  \text{ when }  k=0;\\
        m^2r+2mk+\max \{k^2-(m+k),0)\}\leq t^2\leq m^2r+2mk+\frac{k^2}{r},\\
        \hfill \text{ when } k< 0;
        \end{cases}$
     \item[(d)] $t^2-(m+k)^2-(r-1)m^2-3t+mr+k \geq -2$.
   \end{itemize}
\end{theorem}
   A potential curve $C$ from the linear system on $\mathbb{P}^2$
   corresponding to numbers $t,m,k$ would give the ratio
   $\frac{L\cdot C}{\sum_{i=1}^r \mult_{x_i}(C)}\leq \frac{t}{mr+k}$.
   Thus there is an infinite list of linear systems on $\P^2$, which might
   contain among their elements weakly-submaximal curves. In order to give a lower
   estimate $\alpha$ for the multi-point Seshadri constant in
   $r$ general points, we need to prove that these linear systems
   connected with the numbers $(t,m,k)$ are empty for
   $\frac{t}{mr+k}<\alpha $. Observe that for each
   $\alpha<\sqrt{\frac{1}{r}}$ there is only a finite set of systems to
   check.

   The emptiness of the above systems is proved applying methods of
   \cite{Dum07}. More precisely one uses the algorithm called
   {\sc NSsplit}, which has proved up to date to be the most efficient for checking
   non-speciality (in particular emptiness) of linear systems defined on
   $\P^2$ by vanishing with given multiplicities at a
   number of points in very general position. As this is not directly
   connected with the study of Seshadri constants we omit details and
   refer to the original paper for a precise description of the
   algorithm.

   Recall that for all $r$ which are squares, the Nagata conjecture holds
   and thus gives the exact value of the Seshadri constant. For integers $r$ with $10\leq r\leq
   32$ which are not squares, using the above method M. Dumnicki
   obtained the following table of estimates:
$$
\begin{array}{ccccc}\hline
r  & \text{lower} & \text{approximate} & \text{non-checked } &\text{conjectured} \\[0.2ex]
   & \text{estimate} & \text{value} & \text{system} &\text{approximate value} \\[0.8ex]
   \hline
10 & \frac{313}{990} & \simeq 0.3161616162    & L(313;99^{10})   &   \simeq 0.3162277660 \\[0.8ex]
11 & \frac{242}{803} & \simeq 0.3013698630    & L(242;73^{11})   &  \simeq 0.3015113446 \\[0.8ex]
12 & \frac{277}{960} & \simeq 0.2885416667    & L(277;80^{12})   &  \simeq 0.2886751346 \\[0.8ex]
13 & \frac{602}{2171} & \simeq 0.2772915707   & L(602;167^{13})  & \simeq 0.2773500981 \\[0.8ex]
14 & \frac{389}{1456} & \simeq 0.2671703297   & L(389;104^{14}) &  \simeq 0.2672612419 \\[0.8ex]
15 & \frac{484}{1875} & \simeq 0.2581333333   & L(484;125^{15})  &  \simeq 0.2581988897 \\[0.8ex]
17 & \frac{305}{1258} & \simeq 0.2424483307   & L(305;74^{17})   &  \simeq 0.2425356250 \\[0.8ex]
18 & \frac{369}{1566} & \simeq 0.2356321839   & L(369;87^{18})  &  \simeq 0.2357022604 \\[0.8ex]
19 & \frac{741}{3230} & \simeq 0.2294117647   & L(741;170^{19}) &  \simeq 0.2294157339 \\[0.8ex]
20 & \frac{796}{3560} & \simeq 0.2235955056   & L(796;178^{20}) &  \simeq 0.2236067977 \\[0.8ex]
21 & \frac{1865}{8547} & \simeq 0.2182052182  & L(1865;407^{21}) &  \simeq 0.2182178902 \\[0.8ex]
22 & \frac{924}{4334} & \simeq 0.2131979695   & L(924;197^{22})  &  \simeq 0.2132007164 \\[0.8ex]
23 & \frac{585}{2806} & \simeq 0.2084818247   & L(585;122^{23})  &  \simeq 0.2085144141 \\[0.8ex]
24 & \frac{965}{4728} & \simeq 0.2041032149   & L(965;197^{24})  &  \simeq 0.2041241452 \\[0.8ex]
26 & \frac{622}{3172} & \simeq 0.1960907945   & L(622;122^{26})  &  \simeq 0.1961161351 \\[0.8ex]
27 & \frac{956}{4968} & \simeq 0.1924315620   & L(956;184^{27}) &  \simeq 0.1924500897 \\[0.8ex]
28 & \frac{2434}{12880} & \simeq 0.1889751553 & L(2434;460^{28}) &  \simeq 0.1889822365 \\[0.8ex]
29 & \frac{2364}{12731} & \simeq 0.1856884769 & L(2364;439^{29}) &  \simeq 0.1856953382 \\[0.8ex]
30 & \frac{2388}{13080} & \simeq 0.1825688073 & L(2388;436^{30}) &  \simeq 0.1825741858 \\[0.8ex]
31 & \frac{10729}{59737} & \simeq 0.1796039306& L(10729;1927^{31})&  \simeq 0.1796053020 \\[0.8ex]
32 & \frac{1137}{6432} & \simeq 0.1767723881  & L(1137;201^{32}) &  \simeq 0.1767766953 \\[0.8ex]
\hline
\end{array}
$$

   In the fourth column there is included the list of systems
   not yet proven to be empty. The notation $L(d,m^r)$ stands
   for the system of curves of degree $d$ passing with multiplicity $m$
through each
   of $r$ general points.

\begingroup
\setcounter{theorem}{0}
\renewcommand{\thetheorem}{\thesubsection.\arabic{theorem}}
\renewcommand{\theequation}{\thesubsection.\arabic{theorem}}

\section{Toric varieties} 
   Toric varieties carry strong local constraints, due to the torus action.
   The behavior of Seshadri constants at a given number of points
   is bounded by  the maximal generation of jets at that number of points.
   Equivalently, the Seshadri criterion of ampleness, Theorem \ref{Seshcrit},
   generalizes to a criterion on the generation of multiple higher order jets.
   Moreover, estimates on  local positivity can be explained by properties of an associated convex integral polytope.

   Some of the results reported in this section are  contained in \cite{DR99}
   to which we refer for more details regarding  proofs.
   Some background on toric geometry will be explained, but we refer to \cite{Fu} for more.

\setcounter{theorem}{0}
\renewcommand{\thetheorem}{\thesubsection.\arabic{theorem}}
\renewcommand{\theequation}{\thesubsection.\arabic{theorem}}

\subsection{Toric Varieties and polytopes}
   Let $X$ be a non-singular  toric variety of dimension $n$ and $L$ be an ample line bundle on $X.$
   We identify the torus $T$, acting on $X$, with $N\tensor\C$, for an $n$-dimensional lattice $N\cong \Z^n$.
   The geometry of $X$ is completely described by a fan $\Delta\subset N$.
   In particular the $n$-dimensional cones in the fan, $\sigma_1,...,\sigma_l$, define affine patches:
   $$X=\bigcup_{i=1}^l U_{\sigma_i}.$$
   Since $X$ is non-singular, every cone $\sigma\in\Delta$
   is given by  $\sigma_j=\sum_{i=1}^n \R_+ n_i,$ where the $\{n_i\}$ form
   a lattice basis for $N$.
   Let $\Delta(s)$ denote the set of cones of $\Delta$ of dimension $s$.
   Every $n_i\in \Delta(1)$ is associated to a divisor  $D_i$.

   The Picard group of $X$ has finite rank and it is generated by the divisors $D_i$:
   $$Pic(X)=\bigoplus_{i=1}^d \Z< D_i>.$$
   Hence we can write $L=\sum_{i=1}^d a_i D_i$.

   The pair $(X,L)$ defines a convex, $n$-dimensional, integral polytope in the lattice $M$ dual to $N$:
   $$P=P_{(X,L)}=\{v\in M\, |\, <v,n_i>\geq a_i\}.$$
   We will denote by $P(s)$ the set of faces of $P$ of dimension $s$.
   In particular $P(0)$ is the set of vertices and
   $P(n-1)$ is the set of facets. We denote by $|F|$  the number of lattice points on the face $F$.
   There is the following one-to-one correspondence:
 \[\begin{array}{ccccc}
 \sigma\in\Delta(n)&\Leftrightarrow&v(\sigma)\in P(0)&\Leftrightarrow& x(\sigma)\text{ fixed point }\\
n_i\in\Delta(1)&\Leftrightarrow&F_i\in P(n-1)&\Leftrightarrow& D_i\text{ invariant divisors }\\
 \rho\in\Delta(n-1)&\Leftrightarrow&e_{\rho}\in P(1)&\Leftrightarrow& C_{\rho}\text{ invariant curve }
 \end{array}\]

   Moreover $C_{\rho}\cong \P^1$ for every $ \rho\in\Delta(n-1)$.

   Recall also that the toric variety $X$ being non-singular
   is equivalent to the polytope $P$ being Delzant, i.e., satisfying the following two properties:
\begin{itemize}
   \item there are exactly $n$ edges originating from each vertex;
   \item for each vertex, the first integer points on the edges form a lattice basis.
\end{itemize}

   By the length of an edge $e_\rho$ we mean $|e_\rho|-1$.

   Recall that $H^0(X,L)\cong \oplus_1^{|P\cap M|} \C$.
   By $\{s(m)\}_{m\in P\cap M}$ we denote a  basis for $H^0(X,L)$.

   \setcounter{theorem}{0}
\renewcommand{\thetheorem}{\thesubsection.\arabic{theorem}}
\renewcommand{\theequation}{\thesubsection.\arabic{theorem}}

\subsection{Torus action and Seshadri constants}
   Seshadri constants on non-singular toric varieties are particularly easy to estimate because of
   an explicit  criterion for the generation of $k$-jets.

   Proposition \ref{lowjet} tells us that as soon as we are able to estimate the highest degree of jets generated by
   all multiples of $L$ we can compute the Seshadri constant of $L$ at any point $x\in X$.

   We begin by showing that the generation of jets at the fixed points is detected by the size of the associated polytope.

\begin{lemma}[Generation of jets on toric varieties]\label{fixed}
   Let $x(\sigma)$ be a point fixed by the torus action.
   A line bundle $L$ generates $k$-jets and not $(k+1)$-jets  at $x(\sigma)$,
   if and only if all the edges of $P$ originating from $x(\sigma)$ have length at least $k$,
   and there is at least one edge of length $k$.
\end{lemma}
\begin{proof}
   Let $x(\sigma)$ be a fixed point. We can choose local coordinates $(x_1,...,x_n)$,
   in the affine patch $U_{\sigma}\cong\C^n$Ê such that $x(\sigma)=0$.
   After choosing the lattice basis $(m_1,...,m_n)$, given by the first lattice points on the edges from $x(\sigma)$
   the map
   $$
   \phi_{x(\sigma)}:H^0(X,L)\to H^0(L\tensor\calo_X/m_{x(\sigma)}^{k+1})
   $$
   is defined by
   $$
   s(m=\sum b_i m_i)\mapsto (\Pi x_i^{b_i}|_{x=0}, \ldots, \frac{\partial \Pi x_i^{b_i}}{\partial x_i}|_{x=0}, \ldots, \frac{\partial^k \Pi x_i^{b_i}}{\partial ^{k_1}x_{i_1}\ldots\partial^{k_j} x_{i_j} }|_{x=0}, \ldots).
   $$
   This map is indeed surjective if and only if,
   in the given basis,  $(b_1,...,b_n)\in P\cap M$ for $\sum b_i=k$.
   By convexity this is equivalent to the length of  the edges of $P$ originating from $x(\sigma)$ being at least $k$.
\end{proof}
   Observing that $P_{tL}=tP$, the above criterion gives the exact value of Seshadri constants  at the fixed points. Let
   $$
   s(P,\sigma)=\min_{v(\sigma)\in e_\rho}\{|e_\rho|-1\}.
   $$
\begin{corollary}[Seshadri constants at torus fixed points]
   $$\epsilon(L, x(\sigma))=s(P,\sigma).$$
\end{corollary}
\begin{proof}
   Theorem \ref{fixed} gives that $kL$ generates exactly $ks(P,\sigma)$-jets at $x(\sigma)$.
   Proposition \ref{lowjet} gives then that $\epsilon(L, x(\sigma))=s(P,\sigma)$.
\end{proof}

   Using this criterion we cannot  give an exact estimate at every point in $X$,
   but we can conclude that toric varieties admit a converse of Proposition \ref{lowjet},
   which can be interpreted combinatorially via the associated polytope.

\begin{theorem}[A jet generation criterion]
   A line bundle $L$ generates $k$-jets at every point $x\in X$
   if and only if all the edges of $P$ originating from $v(\sigma)$ have length at least $k$,
   for all vertices $v(\sigma)\in P(0).$
\end{theorem}
\begin{proof}
   Since the map $\phi_{x}:H^0(X,L)\to H^0(L/m_{x}^{k+1})$ is equivariant, the subset
   $$C=\{x\in X\, |\, Coker(\phi_x)\neq\emptyset\}$$
   is an invariant closed subset of $X$, hence it is proper.

   A line bundle $L$ fails to generate $k$-jets on $X$
   if and only if there is an $x\in X$ such that $Coker(\phi_x)\neq\emptyset$.
   In this case $C\neq\emptyset$ and thus, by the Borel fixed point theorem
   $C^T\neq\emptyset$, where $C^T$ denotes the set of fixed points in $C$.
   We conclude that $L$ fails to generate $k$-jets on $X$ if and only if
   $L$ fails to generate $k$-jets at some fixed point  $x(\sigma)\in X$.
   Lemma  \ref{fixed} implies the assertion.\end{proof}

\begin{corollary}[Higher order Seshadri criterion]
   The Seshadri constant satisfies $\epsilon(L)\geq s$
   if and only if all the edges of $P$ originating from $v(\sigma)$ have length at least $s$,
   for all vertices $v(\sigma)\in P(0)$.
\end{corollary}
\begin{proof}
   If the edges of $P$ originating from $v(\sigma)$ have length at least $k$,
   for all vertices $v(\sigma)\in P(0)$, then the line bundle $kL$ is $ks$-jet ample
   for all $s\geq 1$, at all points $x\in X$. Proposition \ref{lowjet} gives then  $\epsilon(L)\geq s$.

   If $\epsilon(L)\geq s$, then $\epsilon(L, x(\sigma))\geq s$,
   for each fixed point $x(\sigma)$.
   It follows that, for all $(n-1)$ dimensional cones $\rho$  in $\sigma$,
   $$L\cdot C_\rho \geq s\cdot m(C_\rho)\geq  s,$$
   because $m_{x(\sigma)}(C_\rho)=1.$

   The property $L\cdot C_\rho \geq s$ for every $\rho\subset \sigma$
   and for all $\sigma\in\Delta(n)$ is equivalent to
   the edges of $P$ originating from $v(\sigma)$ having length at least $s$,
   for all vertices $v(\sigma)\in P(0),$ see \cite[3.5]{DR99}.
\end{proof}

   We easily conclude that:

\begin{corollary}[Global Seshadri constants are integers]
   $$\epsilon(L)=\min_{\sigma\in \Delta}s(P,\sigma).$$ In particular $\epsilon(L)$ is always an integer.
\end{corollary}

  \begin{example}\rm
   The polarized variety associated to the polytope
  \begin{center}
{}{}\scalebox{0.80}{\includegraphics{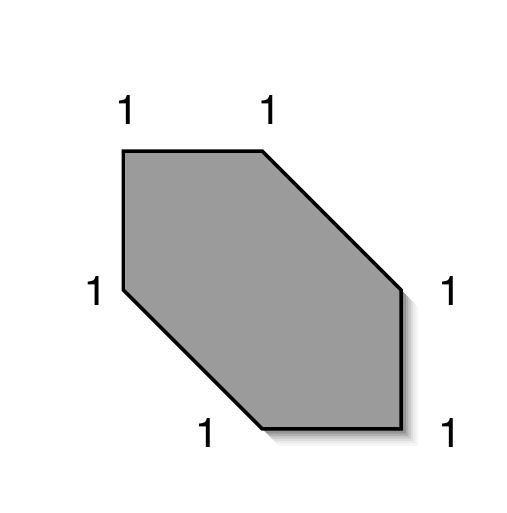}}[Fig. 1]\label{polytope}
 \end{center}
   is $(X,L)=(Bl_{P_1,P_2,P_3}(\P^2), \pi^*(\calo_{\P^2}(3)-E_1-E_2-E_3))$,
   where $\pi$ is the blow up of $\P^2$ at the three points fixed by the torus action
   and $E_i$ are the corresponding exceptional divisors.
   We see that  $s(P,\sigma)=1$ at all vertices, which shows that $\epsilon(L, x(\sigma))=1$ at the six fixed points.
   A local calculation shows that
  \[\begin{array}{cc}
  \epsilon(L, x)=2&\text{ for all }x\in X\setminus \cup_{\rho\in\Delta(1)} C_\rho\;;\\
  \epsilon(L, x)=1&\text{ for all }x\in \cup_{\rho\in\Delta(1)} C_\rho\;;\\
  \epsilon(L)=1\;.
 \end{array}\]
\end{example}

\endgroup
\begingroup

\section{Slope stability and Seshadri constants} 

\renewcommand\O{{\cal O}}
\newcommand\parag[1]{\par\medskip\noindent\textbf{#1}}
\let\tilde=\widetilde

   In \cite{RosTho07}, Ross and Thomas studied various notions of
   stability for polarized varieties, each of which leads to a
   concept of slope for varieties and subschemes. Our purpose in
   this section is to briefly touch upon this circle of ideas,
   and to see how Seshadri constants enter the picture. In order
   to be as specific as possible, we restrict attention to the
   concept of \emph{slope stability}; by way of example we
   present a~result on exceptional divisors of high genus from
   \cite{PanRos}.

   \parag{Slope of a polarized variety.}
   Let $X$ be a smooth projective variety and let $L$ be an ample
   line bundle on $X$. We consider the Hilbert polynomial
   $$
      P(k)=\chi(kL)=a_0k^n+a_1k^{n-1}+O(k^{n-2})
   $$
   and define the \textit{slope} of $(X,L)$ to be the rational
   number
   $$
      \mu(X,L)=\frac{a_1}{a_0}\,.
   $$
   In terms of intersection numbers, we have by Riemann-Roch
   $a_0=\frac1{n!}L^n$
   and $a_1=-\frac1{2(n-1)!}K_X\cdot L^{n-1}$,
   and therefore
\addtocounter{theorem}{1}
   \begin{equation}\label{slope-formula}
      \mu(X,L)=-\frac{n K_X\cdot L^{n-1}}{2L^n}\,.
   \end{equation}

   \parag{Slope of a subscheme.}
   Consider next a proper closed subscheme $Z\subset X$. On the
   blowup $f:Y\to X$ along $Z$ with the exceptional divisor $E$,
   the $\Q$-divisor $f^*L-xE$ is ample for $0< x<\eps(L,Z)$. Here
   $\eps(L,Z)$ is the Seshadri constant of $L$ along $Z$ (see
   Definition~\ref{defsubscheme}). There are polynomials
   $b_i(x)$ such that
   $$
      \chi(k(f^*L-xE))=b_0(x)k^n+b_1(x)k^{n-1}+O(k^{n-2})
      \qquad\mbox{for $k\gg0$ with $kx\in\N$.}
   $$
   One now sets $\tilde a_i(x)=a_i-b_i(x)$ and defines the
   \emph{slope of $Z$} with respect to a given real number $c$
   (and with respect to the polarization $L$) to be
   $$
      \mu_c(\calo_Z,L)=\frac{\int_0^c[\tilde a_1(x)+\frac12\frac{d}{dx}\tilde a_0(x)] \,dx}{\int_0^c \tilde a_0(x)\,dx} \,.
   $$
   When $Z$ is a divisor on a surface, then by Riemann-Roch one
   has
\addtocounter{theorem}{1}
   \begin{equation}\label{subscheme-slope-formula}
      \mu_c(\mathcal{O}_Z,L)=\frac{3(2L\cdot Z-c(K_X\cdot
      Z+Z^2))}{2c(3L\cdot Z-cZ^2)} \,.
   \end{equation}

   \parag{Slope stability.}
   One says that $(X,L)$ is \emph{slope semistable with respect
   to $Z$}, if
   $$
      \mu(X,L)\le\mu_c(\calo_Z,L) \qquad\mbox{for $0<c\le\eps(L,Z)$.}
   $$
   In the alternative case, one says that \emph{$Z$ destabilizes
   $(X,L)$}.
   (We will see below that in order to show that a
   certain subscheme is destabilizing, the crucial point is to find
   an appropriate $c$ in the range that is determined
   by the
   Seshadri constant of $Z$.)
   One checks that if the condition of semistability is
   satisfied, then it is also satisfied for $mL$ instead of $L$.
   So the notion extends to $\Q$-divisors.

\begin{remark}\rm
   The condition that a certain subscheme $Z$ destabilizes
   $(X,L)$ may be seen as a bound on the Seshadri constant
   $\eps(L,Z)$: For instance, when $X$ is a surface,
   then by \eqnref{slope-formula} and
   \eqnref{subscheme-slope-formula} a divisor $Z$
   destabilizes $(X,L)$ iff the inequality
   $$
      \frac{-K_X\cdot L}{L^2}>\frac{3(2L\cdot Z-c(K_X\cdot
      Z+Z^2))}{2c(3L\cdot Z-cZ^2)}
   $$
   holds
   for some number $c$ with $0<c<\eps(L,Z)$.
\end{remark}

   Interest in slope
   stability stems in part from the fact that it gives a concrete
   obstruction to other geometric conditions -- for instance it
   is implied by the existence of constant scalar curvature
   K\"ahler metrics (see \cite{RosTho06}). It is therefore
   natural to ask which varieties are slope stable, and to
   study the geometry of destabilizing subschemes. For the
   surface case, Panov and Ross have addressed this problem in
   \cite{PanRos}. They show that if a polarized surface $(X,L)$
   is slope unstable, then
   \begin{itemize}\itemsep=0cm\parskip=0cm
   \item
      there is a \emph{divisor} $D$ on $X$ such that $D$
      destabilizes $(X,L)$, and
   \item
      if a divisor $D$ destabilizes $(X,L)$, then $D$ is not nef.
      If in addition $X$ has non-negative Kodaira dimension, then
      $D^2<0$.
   \end{itemize}
   In the other direction, they show

\begin{theorem}
   Let $X$ be a smooth projective surface containing an effective
   divisor $D$ with $p_a(D)\ge 2$ whose intersection matrix is
   negative definite. Then there is a polarization $L$ on $X$
   such that $(X,L)$ is slope unstable.
\end{theorem}

   Note that the theorem does not claim that the \emph{given}
   divisor destabilizes. As the proof below shows, it is rather
   the \emph{numerical cycle of $D$} that is claimed to
   destabilize. (Recall that the numerical cycle of a divisor
   $D=\sum_i  d_iD_i$ -- also
   called \emph{fundamental cycle} in the literature -- is the
   smallest non-zero effective (integral) divisor $D'=\sum_i
   d_i'D_i$ such that
   $D'\cdot D_i\le 0$ for all $i$. For its existence and
   uniqueness see \cite[Sect.~4.5]{Rei97}.)

\begin{proof}
   Write $D=\sum_{i=1}^m d_iD_i$ with irreducible divisors $D_i$
   and integers $d_i>0$. One reduces first to the case where
\addtocounter{theorem}{1}
   \begin{equation}\label{reduction-negative}
      D\cdot D_i\le 0 \qquad\mbox{for $i=1,\dots,m$.}
   \end{equation}
   To get \eqnref{reduction-negative}, replace $D$ by its
   numerical cycle $D'$. Then work by Artin \cite{Art66}, Laufer
   \cite{Lau77}, and N\'emethi \cite{Nem99} implies that the
   inequality $p_a(D')\ge 2$ follows from the hypothesis
   $p_a(D)\ge 2$.

   Assuming now \eqnref{reduction-negative}, we fix an ample
   divisor $H$ and we construct a divisor
   $$
      L_0:=H+\sum_i q_iD_i
   $$
   with rational coefficients $q_i$ such that $L_0\cdot D_i=0$
   for all $j$. Such a divisor exists uniquely thanks to the
   negative definiteness of the intersection matrix of $D$. As
   the inverse of this intersection matrix has all entries $\le
   0$ (cf. \cite[Lemma~4.1]{BKS}), it follows that $q_i\ge 0$ for
   all $i$. Since $H$ is ample, we actually have $q_i>0$ for all
   $i$. Letting now $\eps=\min_i\{q_i/d_i\}$, we claim that
\addtocounter{theorem}{1}
   \begin{equation}\label{L_0-cD nef}
      L_0-cD \mbox{ is nef for $0\le c\le\eps$.}
   \end{equation}
   In fact, we have $(L_0-cD)\cdot D_i\ge 0$ thanks to
   \eqnref{reduction-negative}, and for curves $C$ different from
   the $D_i$ we have
   $(L_0-cD)\cdot C=(H+\sum_i(q_i-cd_i)D_i)\cdot C> 0$. The proof is now
   completed by showing that
\addtocounter{theorem}{1}
   \begin{equation}\label{claim-destab}
      D \mbox{ destabilizes } L_s:=L_0+sH \mbox{ for small
      $s>0$.}
   \end{equation}
   To see \eqnref{claim-destab}, note first that
   $L_s-cD$ is clearly ample for $0\le c\le\eps$ and for every $s>0$,
   hence $\eps(L_s,D)\ge\eps$.
   We have\footnote{In the two subsequent displayed equations the
   expressions for $\mu(X,L_0)$ and $\mu_c(\calo_D,L_0)$ from
   \eqnref{slope-formula} and \eqnref{subscheme-slope-formula}
   are used formally even though $L_0$ is not ample.
   The formulas
   \eqnref{slope-formula} and \eqnref{subscheme-slope-formula}
   may be viewed as the definitions of $\mu$ and
   $\mu_c$ in this case.
   From this perspective,
   the point is only that $\mu(X,L_s)$ tends to $\mu(X,L_0)$
   when $s\to 0$, and similarly for $\mu_c$.}
   $$
      \mu(X,L_0)=\frac{-K_X\cdot L_0}{L_0^2}\,,
   $$
   which is finite because $L_0^2=H\cdot L_0\ge H^2>0$, and we have
   $$
      \mu_c(\calo_D,L_0)=\frac{3(2L_0\cdot D-c(K_X\cdot
      D+D^2))}{2c(3L_0\cdot D-cD^2)}
      =\frac{3(2p_a(D)-2)}{2cD^2}\,.
   $$
   As $D^2<0$, and
   thanks to the hypotheses on $p_a(D)$, the latter tends to
   $-\infty$ for $c\to 0$. We can therefore choose a $c$ with
   $0<c<\eps$ such that $\mu_c(\calo_D,L_0)<\mu(X,L_0)$. Choosing
   now $s>0$ small enough, we still have
   $\mu_c(\calo_D,L_s)<\mu(X,L_s)$ while $c<\eps(L_s,D)$, and this
   proves \eqnref{claim-destab}.
\end{proof}

\endgroup


\begingroup

\section{Seshadri constants on surfaces}

\setcounter{theorem}{0}
\renewcommand{\thetheorem}{\thesubsection.\arabic{theorem}}
\renewcommand{\theequation}{\thesubsection.\arabic{theorem}}

\subsection{Bounds on arbitrary surfaces}

Not surprisingly, the case of surfaces is the case that has been
studied the most. We will in this section present some of the
known results. So let $S$ be a smooth projective surface, $L$ an
ample line bundle on $S$ and $x$ any point on $S$.

First of all note that $\varepsilon(L,x) \leq \sqrt{L^2}$ by
Proposition \ref{upperbound} and that $\varepsilon(L,x)$ is rational
if strict inequality holds, by Theorem \ref{thm:submax}. In fact one
has the following improvement due to Oguiso \cite[Cor. 2]{Ogu02}
(see also \cite[Lemma 3.1]{Sze01}):

\begin{theorem}[Submaximal global Seshadri constants] \label{thm:submax2}
   Let $(S,L)$ be a smooth polarized surface. If
$\varepsilon(L) < \sqrt{L^2}$, then there is a point $x \in S$ and a
curve $x \in C \subset S$ such that
$\varepsilon(L)=\varepsilon(L,x)=\frac{L\cdot C}{\mult_xC}$.

In particular, $\varepsilon(L)$ is rational unless $\varepsilon(L) =
\sqrt{L^2}$ and $\sqrt{L^2}$ is irrational.
\end{theorem}

In fact, in \cite{Ogu02}, Oguiso studies Seshadri constants of a
family of surfaces $\{f:\mathcal{S} \to B, \mathcal{L} \}$, where
$f$ is a surjective morphism onto a non-empty Noetherian scheme $B$,
$\mathcal{L}$ is an $f$-ample line bundle and the fibers $(S_t,L_t)$
are polarized surfaces of degree $L_t^2$ over an arbitrary closed
field $k$. He proves \cite[Cor.5]{Ogu02}

\begin{theorem}[Lower semi-continuity of Seshadri constants]
\label{thm:lsc}
(1) For each fixed $t \in B$, the function $y = \varepsilon(x) :=
\varepsilon (L_{t}, x)$ of $x \in S_{t}$ is lower semi-continuous
with respect to the Zariski topology of $S_{t}$.

(2) The function $y = \varepsilon(t) := \varepsilon (S_t,L_{t})$ of
$t \in B$ is lower semi-continuous with respect to the Zariski
topology of $B$.
\end{theorem}

A nice visualization of this result is provided by the global
Seshadri constants of quartic surfaces in Theorem \ref{thm:k3q}
below: They are mostly constant but jump down along special loci in
the moduli.

Much attention has been devoted to the study of (the existence of)
submaximal curves (cf. Definition \ref{submaxcurve}), that is,
curves $C$ for which $\frac{L\cdot C}{\mult_xC} < \sqrt{L^2}$ and to
possible values of
$\varepsilon_{C,x}:=\frac{L\cdot C}{\mult_xC}$. In \cite[Thm.
4.1]{Bau99}, the degree of submaximal curves at a \emph{
very general point} $x$ is bounded by showing that
    $$ L \cdot C < \frac{L^2}{\sqrt{L^2} - \varepsilon_{C,x}}\;. $$
    Moreover, \cite[Prop. 5.1]{Bau99} provides also bounds on the number of
curves satisfying
    $\frac{L\cdot C}{\mult_xC} < a$ for any $a \in \R^+$.
    These results have been generalized to
    multi-point Seshadri constants by Ro\'e and the third named author in
    \cite[Lemma 2.1.4 and Thm. 2.1.5]{HR08}. The main result of \cite{HR08}
implies that when the
    Seshadri constant is submaximal, then the
    set of potential Seshadri curves is finite.

    As for lower bounds, we recall the following result obtained
    in \cite[Thm. 3.1]{Bau99} in terms of the quantity
    $\sigma(L)$, which is
    defined as
    $$\sigma(L):= \frac{1}{\eps(L,K_S)} =\mbox{min} \; \{s \in \R \; | \;
\calo_S(sL-K_S) \; \mbox{is nef} \}\;. $$

\begin{theorem}[Lower bound in terms of canonical slope]
    Let $(S,L)$ be a~smooth polarized surface. Then
    \[ \varepsilon(L) \geq \frac{2}{1+ \sqrt{4\sigma(L)+13}}. \]
\end{theorem}

Note that for $(S,L)=(\P^2, \calo_{\P^2}(1))$ equality holds, as
$\sigma(L)=-3$ and $\varepsilon(L)=1$. Also note that for surfaces
of Kodaira dimension zero, $\sigma(L)=0$ and the theorem yields
$\varepsilon(L) \geq 0,434 \ldots$, whereas the optimal bound is
$\varepsilon(L) \geq \frac{1}{2}$ on an Enriques or $K3$ surface
(see Theorem \ref{thm:lbenr} and the beginning of \S \ref{sect:k3})
and $\varepsilon(L) \geq \frac{4}{3}$ on a simple abelian surface (see
Theorem \ref{thm:spablb}(a)). Moreover, in the case of a smooth
quartic in $\P^3$, the value of $\varepsilon(L)$ strongly depends on
the geometry of $S$ (see Theorem \ref{thm:k3q} below), so that one
cannot expect that $\sigma(L)$ alone fully accounts for the
behaviour of the Seshadri constant.

When the Picard number $\rho(S)$ of the surface is one, we have the
following optimal result \cite[Theorem 7]{Sze07},
yielding an answer to Question \ref{lowerx}.

\begin{theorem}[Effective lower bound on surfaces with $\rho(S)=1$]
\label{eff-bound-sur}
    Let $S$ be a smooth projective surface with $\rho(S)=1$ and let
    $L$ be an ample line bundle on $S$. Then for any point $x\in S$
    \begin{itemize}
    \item[(S)]
    $\eps(L,x)\geq 1$ if $S$ is not of general type and
    \item[(G)]
    $\eps(L,x)\geq\frac{1}{1+\sqrt[4]{K_S^2}}$ if $S$ is of
    general type.
    \end{itemize}
    Moreover both bounds are sharp.
\end{theorem}

Equality in (S) is for example attained for $S=\P^2$ and $L =
\calo_{\P^2}(1)$. Equality in (b) is attained in the following
example (see \cite{Sze07} or \cite[Example 1.2]{BauSze08}):

\begin{example}\label{exa:efflb}\rm
    Let $S$ be a smooth surface of general type with $K_S^2=1$, $p_g(S)=2$
and $\rho(S)=1$. An example of such a surface is a
    general surface of degree $10$ in the weighted projective space
    $\P(1,1,2,5)$. Then, $\rho(S)=1$ by a result of Steenbrink
    \cite{Ste82}. Moreover, by adjunction $K_S^2=1$ and sections of
    $K_S$ correspond to polynomials of degree one in the weighted
    polynomial ring on $4$ variables. Thus $p_g(S)=2$, cf. also
    \cite{Ste82}.

    We now claim that there exists an $x \in S$ such that $\eps(K_S,x)
    =\frac{1}{2}$. Indeed, the curves in the pencil $|K_S|$ cannot carry
    points of multiplicity $>2$ since they have arithmetic genus two and
    cannot all be smooth, which can be seen directly computing the
topological
    Euler characteristic of $S$.
\end{example}

    Looking back at the examples of Miranda mentioned in \S \ref{low-bou},
we see that the lower bound
    $\eps(L,x)\geq\frac{1}{1+\sqrt[4]{K_S^2}}$ holds. One could
    therefore hope that this (or some ``nearby'' number) would serve as
    a lower bound on arbitrary surfaces. In fact, there is a conjectural
    effective lower
    bound for all minimal surfaces \cite[Question]{Sze07}:
\begin{question}[Conjectural effective lower bound on
surfaces]\label{eff-low-con}
    For any mini\-mal surface $S$, an ample line bundle $L$ and $x\in
    S$ is it true that
    $$\eps(L,x)\geq\frac{1}{2+\sqrt[4]{|K_S^2|}}\;?$$
\end{question}

The appearance of $2$ in the denominator is in fact necessary due to
Enriques and $K3$ surfaces carrying ample line bundles with
$\eps(L,x)=\frac{1}{2}$, see \S \ref{sect:enriques} and \S
\ref{sect:k3}.

    Better lower bounds are known if $x$ is a (very) general point.
    We observed already in \ref{sesfun} that for $x$ away from a countable
union of Zariski closed subsets
    $\eps(L;x)$ is constant. We denote its value by $\eps(L;1)$. A
fundamental result of Ein and Lazarsfeld,
    which we recall in Theorem \ref{ellower} states that on surfaces
    $$\eps(L;1)\geq 1\;.$$
    In fact, if $L^2 >1$, they
    proved, cf. \cite[Theorem]{EinLaz93} that $\eps(L,x) \geq 1$ for all
    but \emph{finitely many} points on $S$. This result was improved by Xu
\cite[Thm. 1]{Xu95}:

\begin{theorem}[Xu's lower bound on surfaces]\label{thm:xu}
    Let $(S,L)$ be a smooth polarized surface. Assume that, for a given
integer $a>1$, we have  $L^2 \geq \frac{1}{3}(4a^2-4a+5)$ and $L\cdot C
\geq a$ for every irreducible curve $C \subset S$.

    Then $\eps(L,x) \geq a$ for all $x \in S$ outside of finitely many
curves on $S$.
\end{theorem}

(Note that in fact $\eps(L,x) \geq a$ outside finitely many points
on $S$ if there is no curve $C$ such that $L\cdot C=a$.)

    In the case of Picard number one, Steffens \cite[Prop. 1]{Ste98}
proved:

\begin{theorem}[Steffens' lower bound for $\rho(S)=1$]
\label{thm:steffens}
    Let $S$ be a smooth surface with $\mbox{NS} \; (S) \simeq \Z[L]$. Then
    $$\eps(L;1) \geq \lfloor \sqrt{L^2} \rfloor\;.$$

    In particular, if $\sqrt{L^2}$ is an integer, then
    $\eps(L;1)=\sqrt{L^2}$.
\end{theorem}

    In the case of very ample line bundles, these results have been
    generalized to the case of multi-point Seshadri constants at
    general points in \cite[Thm. I.1]{Harb03}. Recall that
    $\varepsilon(L; x_1, \dots, x_r) \leq  \sqrt{\frac{L^2}{r}}$ by
    Proposition \ref{upperbound}. Given any $c \in \R$, we write
    $\epsilon(L; r) \ge c$ if $\epsilon(L, x_1,\dots, x_r)\ge c$
    holds on a \emph{Zariski-open set} of $r$-tuples of points $x_i$
of $X$. Moreover, let $\varepsilon_{r,l}$ be the maximum element in
the finite set $$\Bigl\{{\lfloor d\sqrt{rl}\rfloor \over
dr}\,\Bigl|\, 1\le d\le \sqrt{{r\over l}}\Bigr\}
\cup\Bigl\{{1\over\lceil\sqrt{{r\over l}}\rceil}\Bigr\}\cup
\Bigl\{{dl\over\lceil d\sqrt{rl}\rceil}\,\Bigl| \,1\le d
\le\sqrt{{r\over l}}\Bigr\}.$$ (Note that
$\varepsilon_{r,l}=\sqrt{lr}$ if $l <r$ and $rl$ is a square, cf.
\cite[Prop. III.1(b)(i)]{Harb03}.)

Then, we have the following result:

\begin{theorem}[Lower bound for multi-point Seshadri constants]
    Let $S$ be a~smooth surface and $L$ a very ample line bundle on $S$.
Set $l:=L^2$.

    Then $\sqrt{l/r}\ge \epsilon(L, r)$, and in addition, $\epsilon(L,r)
\ge \varepsilon_{r,l}$
    unless $l\le r$ and $rl$ is a~square, in
    which case $\sqrt{l/r}=\varepsilon_{r,l}$ and
$\epsilon(L,r)\ge\sqrt{l/r}-\delta$ for every positive rational
    $\delta$.
\end{theorem}

    The somewhat awkward statement in the case $rl$ is a square is due
    to the possibility of there being no \emph{open} set of points such
that
    $\epsilon(L, r) = \varepsilon_{r,l}$ in that case. Also note that
    the result holds over an algebraically closed field of any
    characteristic. Over the complex numbers, one obtains a
    generalization of the last statement in Theorem \ref{thm:steffens}:

\begin{theorem}[Maximality of multi-point Seshadri constants]
    Let $S$ be a~smooth surface and $L$ a very ample line bundle on $S$.
    Let $r \in \Z$ be such that $r \geq L^2$ and $\sqrt{rL^2}$ is an
integer.

    Then, for a Zariski-open set of points $(x_1, \ldots, x_r) \in S^r$, we
have
    \[\varepsilon(L; x_1, \dots, x_r) = \sqrt{\frac{L^2}{r}}.\]
\end{theorem}

    More specific results are known when one restricts attention to
    surfaces or line bundles of particular types. In the remainder of
    this section, we will present some of these results.

\setcounter{theorem}{0}
\renewcommand{\thetheorem}{\thesubsection.\arabic{theorem}}
\renewcommand{\theequation}{\thesubsection.\arabic{theorem}}

\subsection{Very ample line bundles} \label{sect:va}

    Consider a smooth projective surface $S$ and a very ample line bundle
    $L$ on $S$. By Proposition \ref{lowjet} we have $\varepsilon(L,x) \geq
1$ for any $x \in S$. Moreover, equality is obviously attained if $S$
contains a line (when embedded by the linear series $|L|$).
It is then natural to ask whether this is the only case where
    $\eps(L)=1$ occurs, and what the next possible values of $\eps(L)$
    for a very ample line bundle are. Both of these questions were answered
in \cite[Theorem 2.1]{Bau99}.

\begin{theorem}[Seshadri constants on embedded surfaces] \label{thm:va}
(a)  Let $S \subset \P^N$ be a smooth surface. Then
    $\eps(\calo_S(1))=1$ if and only if $S$ contains a line.

(b)
    For $d\ge 4$ let $\mathcal{S}_{d,N}$
    denote the space of smooth irreducible
    surfaces of degree $d$ in $\P^N$ that do not contain any lines.
    Then
    $$
       \min \Big \{ \eps(\calo_S(1)) \; | \; S \in \mathcal{S}_{d,N} \Big\}
=
\frac d{d-1}.
    $$

(c)
    If $S$ is a surface in $\mathcal{S}_{d,N}$ and $x\in S$ is a point such
that
    the Seshadri constant $\eps(\calo_S(1),x)$ satisfies the inequalities
    $1<\eps(\calo_S(1),x)<2$,
    then it is of the form
    $$
       \eps(\calo_S(1),x) = \frac ab \ ,
    $$
    where $a,b$ are integers with $3\le a\le d$ and $a/2<b<a$.

(d)
    All rational numbers $a/b$ with $3\le a\le d$ and $a/2<b<a$
    occur as local Seshadri constants of smooth irreducible surfaces
    in $\P^3$ of degree $d$.
\end{theorem}

    The examples in (d) are constructed in the following way: given $a$
    and $b$, one can choose an
    irreducible curve $C_0 \subset \P^2$ of degree $a$ with a point $x$ of
    multiplicity $b$. Further, take a smooth curve $C_1 \subset \P^2$
    of degree $d-a$
    not passing through $x$. Then
    there is a smooth surface $S \subset\P^3$ such that
    the divisor $C_0+C_1$ is a~hyperplane section of $X$  and the
    curve $C$ computing
    $\eps(\calo_S(1),x)$ is a component of the intersection $S\cap T_xX$,
    and therefore $C=C_0$. So one can conclude
    $$
       \eps(\calo_S(1),x)=\frac{L \cdot C_0}{\mult_x C_0}=\frac{a}{b}.
    $$

    Note that in the case of quartic surfaces, exact values have been
    computed in \cite{Bau97}, see Theorem \ref{thm:k3q} below.

\setcounter{theorem}{0}
\renewcommand{\thetheorem}{\thesubsection.\arabic{theorem}}
\renewcommand{\theequation}{\thesubsection.\arabic{theorem}}

\subsection{Surfaces of negative Kodaira dimension}

    The projective plane is discussed in \S \ref{rec-est}. The case of
ruled surfaces has been studied by
    Fuentes Garc{\'i}a in \cite{Fue06}. He explicitly computes the
    Seshadri constants in the case of the invariant $e >0$, cf.
\cite[Theorem 4.1]{Fue06}. In the following we let $\sigma$ and $f$
denote the
    numerical class of a section and a fiber, respectively.

\begin{theorem}[Seshadri constants on ruled surfaces with $e>0$]
\label{thm:e>0}
    Let $S$ be a~ruled surface with the invariant $e>0$ and $A \equiv
a\sigma+bf$ be a nef linear system on $S$. Let $x$ be a point of $S$.
    Then
    $$\eps(A;x)=\left\{\begin{array}{lcl}
       \min\{a,b-ae\} & \mbox{ if } & x\in\sigma,\\
       a & \mbox{ if } & x\not\in\sigma.
       \end{array}\right.$$
\end{theorem}

    In particular, note that $\eps(A;x)$ reaches the maximal value
$\sqrt{A^2}$
    only for $e=1$ and $b=a$, at points $x\not \in \sigma$, or when
    $A\equiv b f$, at any point $x\in S$.

    Furthermore,  Fuentes Garc{\'i}a gives the following bounds when
$e\leq 0$, cf. \cite[Thm. 4.2]{Fue06}:

\begin{theorem}[Seshadri constants on ruled surfaces with $e \leq
0$]\label{thm:e<0}
    Let $S$ be a~ruled surface with the invariant $e\leq 0$ and
    $A \equiv a\sigma+bf$ be a nef linear system on $S$. Let $x$ be a
    point of $S$.
\begin{enumerate}
    \item If  $e=0$ and $x$ lies on a curve numerically equivalent to
$\sigma$,
    then $\varepsilon(A,x)=min\{a,b\}$.

    \item In all other cases $\varepsilon(A,x)=a$ if $b-\frac{1}{2}ae\geq
\frac{1}{2}a$ and
    $$ 2(b-\frac{1}{2}ae)\leq
\varepsilon(A,x)\leq\sqrt{A^2}=\sqrt{2a(b-\frac{1}{2}ae)}.
    $$
    if $0\leq b-\frac{1}{2}ae\leq \frac{1}{2}a$.
\end{enumerate}
\end{theorem}

Of course, Theorem \ref{thm:e>0} and case (1) of Theorem
\ref{thm:e<0} completely determine the Seshadri constants on
rational ruled surfaces (as $e \geq 0$, and in the case $e=0$ there
is always a section passing through a given point $x \in S$). From
these two theorems and some more work in the cases $e=-1$ and $e=0$,
Fuentes Garc{\'i}a is also able to explicitly compute all Seshadri
constants on elliptic ruled surfaces, cf. \cite[Thms. 1.2 and
6.6]{Fue06}. Furthermore, he also constructs ruled surfaces and
linear systems where the Seshadri constant does not reach the upper
bound, but is as close as we wish:

\begin{theorem} \label{thm:close}
Given any $\delta \in \R^+$ and a smooth curve $C$ of genus $>0$,
there is a stable ruled surface $S$, an ample divisor $A$ on $S$ and
a point $x\in S$ such that
$$
\sqrt{A^2}-\delta<\varepsilon(A,x)<\sqrt{A^2}.
$$
\end{theorem}

As for del Pezzo surfaces, Broustet proves the following result, cf.
\cite[Thm. 1.3]{Bro06}. Here, $S_r$ for $r \leq 8$, denotes the blow up
of the plane in $r$ general points $\{p_1, \ldots, p_r\}$. We say that $x
\in S_r$ is in general position if its
image point $p \in \P^2$ is such that the points in the set $\{p_1,
\ldots,p_r,p\}$ are in general position.

\begin{theorem}[Seshadri constants of $-K_S$ on del Pezzo surfaces]
\label{thm:delpezzo} If $r \leq5$, then $\varepsilon(-K_{S_r},x)=2$
if $x$ is in general position and $\varepsilon(-K_{S_r},x)=1$
otherwise.

If $r=6$, then $\varepsilon(-K_{S_6},x)=3/2$ if $x$ is in general
position and $\varepsilon(-K_{S_6},x)=1$ otherwise.

If $r=7$, then $\varepsilon(-K_{S_7},x)=4/3$ if $x$ is in general
position and $\varepsilon(-K_{S_7},x)=1$ otherwise.

If $r=8$, then  $\varepsilon(-K_{S_8},x)=\frac{1}{2}$ in at most
$12$ points lying outside the exceptional divisor, and
$\varepsilon(-K_{S_8},x)=1$ everywhere else.
\end{theorem}

\setcounter{theorem}{0}
\renewcommand{\thetheorem}{\thesubsection.\arabic{theorem}}
\renewcommand{\theequation}{\thesubsection.\arabic{theorem}}

\subsection{Abelian surfaces} \label{sect:abelian}

Let $S$ be an abelian surface and $L$ an ample line bundle on $S$.
By homogeneity, $\varepsilon(L,x)$ does not depend on the $x$
chosen. In particular $\varepsilon(L)=\varepsilon(L,x)$ for any $x
\in S$ and one can compute this number for $x$ being one of the
half-periods of $S$, which is the idea in both \cite{App:Bau98} and
\cite{Bau99}. Furthermore, since $\varepsilon(kL)=k\varepsilon(L)$
for any integer $k >0$, one may assume that $L$ is primitive, that
is, of type $(1,d)$ for some integer $d \geq 1$.  The elementary
bounds for single point Seshadri constants one has from Proposition
\ref{upperbound} and (\ref{ab-bound}) are
\addtocounter{theorem}{1}
\begin{equation} \label{eq:a1}
1 \leq \varepsilon(L) \leq \sqrt{2d}
\end{equation}

In the case of Picard number one, exact values for one-point
Seshadri constants were computed in \cite[Thm. 6.1]{Bau99}. To
state the result, we will need:

\begin{notation}[Solution to Pell's equation]\rm
    In the rest of this subsection we let $(\ell_0, k_0)$ denote the
primitive solution of the diophantine equation
    \[ \ell^2-2dk^2=1, \]
    known as \emph{Pell's equation}.
\end{notation}

\begin{theorem}[Exact values on abelian surfaces with $\rho(S)=1$]
\label{thm:spabrkone}

Let $(S,L)$ be a polarized abelian surface of type $(1,d)$ with
$\rho(S)=1$.

    If $\sqrt{2d}$ is rational, then $\varepsilon(L) = \sqrt{2d}$.

    If $\sqrt{2d}$ is irrational, then
    \[ \varepsilon(L) = 2d \cdot \frac{k_0}{\ell_0}=
\frac{2d}{\sqrt{2d+\frac{1}{k_0^2}}}  \; \; \Big(< \sqrt{2d} \Big). \]
\end{theorem}

In the general case, the lower bound in (\ref{eq:a1}) has been
improved as follows:

\begin{theorem}[Lower bounds on abelian surfaces]
\label{thm:spablb} Let $(S,L)$ be a polarized abelian surface of
type $(1,d)$.

(a) $\varepsilon(L) \geq \frac{4}{3}$ unless $S$ is non-simple 
(a product of elliptic curves).

(b) $\varepsilon(L) \geq \{ \varepsilon_1(L),
\frac{\sqrt{7d}}{2}\}$, where $\varepsilon_1(L)$ is the minimal
degree with respect to $L$ of an elliptic curve on $S$.
\end{theorem}

Here, statement (a) is due to Nakamaye \cite[Thm. 1.2]{Nak96} and
(b) was proved in \cite[Thm. A.1(b)]{App:Bau98}.

Note that (b) yields a better bound than (a) if $d >2$. However, for
$d=2$, (a) is sharp, as equality is attained for $S$ the Jacobian of
a hyperelliptic curve and $L$ the theta divisor on $S$, by
\cite[Prop. 2]{Ste98}.

Furthermore, note that it is inevitable that small values of
$\varepsilon(L)$ occur for non-simple abelian surfaces regardless of
$d$, since for any integer $e \geq 1$, there are non-simple
polarized abelian surfaces $(S,L)$ of arbitrarily high degree $L^2$
containing an elliptic curve of degree $e$.

The upper bound in (\ref{eq:a1}) can be improved in the case of
$\sqrt{2d}$ being irrational, by the following result, see \cite[Theorem
A.1(a)]{App:Bau98}:

\begin{theorem}[Upper bounds on abelian surfaces]
\label{thm:spabub} Let $(S,L)$ be a polarized abelian surface of
type $(1,d)$. If $\sqrt{2d}$ is irrational, then
\[ \varepsilon(L) \leq 2d \cdot \frac{k_0}{\ell_0}=
\frac{2d}{\sqrt{2d+\frac{1}{k_0^2}}}  \; \; \Big(< \sqrt{2d} \Big). \]
\end{theorem}

In particular, together with Theorem \ref{thm:submax2}, this implies

\begin{theorem}[Rationality on abelian surfaces]\label{thm:ratab}
    Seshadri constants of ample line bundles on abelian surfaces are
rational.
\end{theorem}

For the estimates obtained by combining the upper and lower bounds
above for low values of $d$, we refer to \cite[Rmk. A.3]{App:Bau98}.
For more precise results on submaximal curves, we refer to
\cite[Sec. 6]{Bau99}. Also note that results for non-simple abelian
surfaces have been obtained in \cite{BauSch08}.

The case of multi-point Seshadri constants is much harder. In
fact, if one wants to compute the Seshadri constant in $r$ points
for $r>1$, one can no longer assume that the points are general.

By homogeneity, the number $\varepsilon(L;x_1, \ldots, x_r)$ depends
only on the differences $x_i-x_1$ and  we have
\addtocounter{theorem}{1}
\begin{equation}
   \label{eq:a2}
  \varepsilon(L;x_1, \ldots, x_r) \geq \frac{1}{r}
\mbox{min}\; \varepsilon(L;x_i) = \frac{1}{r} \varepsilon(L)
\end{equation}
(Note that the inequality \eqnref{eq:a2} holds on any variety).
In \cite[Proposition 8.2]{Bau99} it is shown that if equality is
attained, then $S$ contains an
elliptic curve $E$ containing $x_1, \ldots, x_r$ and such that $L\cdot E=
r \varepsilon(L; x_1, \ldots, x_r)$.

    Tutaj-Gasi{\'n}ska gave bounds for Seshadri
    constants in half-period points in \cite{Tut04}. In \cite{Tut05} she
gave exact values
    for the case of two half-period points (with a small gap in the
    proof pointed out in \cite[Remark 2.10]{Fue07}). More precisely, in
\cite[Thm. 3]{Tut04} she proves that if
    $e_1, \ldots, e_r$ are among the $16$ half-period points of $S$, then
    $$\eps(L;e_1,\dots, e_r)\left\{\begin{array}{lcl}
       =\sqrt{\frac{2d}{r}} & \mbox{ if } &  \sqrt{\frac{2d}{r}} \in
\Q\;,\\
       \leq 2d \frac{k_0}{\ell_0} & \mbox{ if } &  \sqrt{\frac{2d}{r}} \not
\in \Q\;.
       \end{array}\right.$$

In the case of Picard number one, the results were generalized by a
different method by Fuentes Garc{\'i}a in \cite{Fue07}, who computes
the multi-point Seshadri constants in points of a finite subgroup
of an abelian surface, cf. \cite[Theorem 1.2]{Fue07}.
One of the
corollaries obtained by Fuentes Garc{\'i}a \cite[Corollary 2.6]{Fue07}
is:

\begin{theorem}[Multi-point Seshadri constants on ab. surfaces with
$\rho(S)=1$]\label{thm:mpab}
    Let $(S,L)$ be a polarized abelian surface of type $(1,d)$ with
$\rho(S)=1$ and
    $x_1, \ldots, x_r$ be general points on $S$.

If $\sqrt{\frac{2d}{r}} \in \Q$, then $\varepsilon(L;x_1, \ldots,
x_r)= \sqrt{\frac{2d}{r}}$.

If $\sqrt{\frac{2d}{r}} \not \in \Q$, then $\varepsilon(L;x_1,
\ldots, x_r) \geq 2d \frac{k_0}{\ell_0}$.
\end{theorem}

Moreover, as a direct consequence of \cite[Theorem 1.2]{Fue07}, one
obtains:

\begin{theorem}[Rationality of multi-point Seshadri constants at finite
subgroups]\label{thm:ratabsub}
    The multiple-point Seshadri constants of ample
    line bundles at the points of a finite subgroup of an abelian
    surface are rational.
\end{theorem}

\setcounter{theorem}{0}
\renewcommand{\thetheorem}{\thesubsection.\arabic{theorem}}
\renewcommand{\theequation}{\thesubsection.\arabic{theorem}}

\subsection{Enriques surfaces} \label{sect:enriques}

Let $S$ be an Enriques  surface (by definition then, $h^1(\calo_S)=0$,
$K_S \neq 0$ and $2K_S=0$) and $L$ an ample line bundle on $S$.
One-point Seshadri constants on Enriques surfaces have been studied
in \cite{Sze01}. It is well-known that there is an
effective nonzero divisor $E$ on $S$ satisfying $E^2=0$ (whence $E$
has arithmetic genus $0$) and $E\cdot L \leq \sqrt{L^2}$, see
\cite[Prop. 2.7.1 and Cor. 2.7.1]{Cos-Dol}. As a consequence, taking any
point
$x \in E$, combining with Theorem \ref{thm:submax}, one obtains
\cite[Thm. 3.3]{Sze01}:

\begin{theorem}[Rationality on Enriques surfaces]
\label{thm:ratenr} Let $(S,L)$ be a polarized Enriques surface. Then
$\varepsilon (L)$ is rational.
\end{theorem}

To state the lower bounds obtained in
\cite[Thm. 3.4 and Prop. 3.5]{Sze01},
define the \emph{genus $g$ Seshadri constant of $L$ at $x$} by
\[ \varepsilon_g(L,x):= \inf\frac{L\cdot C}{\mult_xC},\]
where the infimum is taken over all irreducible curves of arithmetic
genus $g$ passing through $x$. (Note that since an abelian surface
does not contain rational curves, this definition is consistent with
the definition of the number $\varepsilon_1(L)$ in Theorem
\ref{thm:spablb}(b).)

\begin{theorem}[Lower bounds on Enriques surfaces]
\label{thm:lbenr} Let $(S,L)$ be a polarized Enriques surface and $x
\in S$ an arbitrary point.

Then
\[\varepsilon (L,x) \geq \mbox{min} \Big \{
\varepsilon_0(L,x),\varepsilon_1(L,x), \frac{1}{4} \sqrt{L^2} \Big \}. \]

Furthermore, $\varepsilon (L,x) <1$ if and only if  there is an
irreducible curve $E$ on $S$ satisfying $p_a(E)=0$, $L\cdot E=1$ and
$\mbox{mult}_x E=2$ (so that $\varepsilon (L,x) =\frac{1}{2}$).
  \end{theorem}

Note that in the special case of the theorem, $L$ cannot be globally
generated, by Proposition \ref{lowspan} or
directly from a fundamental property of line bundles on Enriques
surfaces \cite[Thm. 4.4.1]{Cos-Dol}. In fact, the proof exploits the
characterization of non-globally generated line bundles on Enriques
surfaces. Also note that the special case attains the lower bound in
Question \ref{eff-low-con}.

\setcounter{theorem}{0}
\renewcommand{\thetheorem}{\thesubsection.\arabic{theorem}}
\renewcommand{\theequation}{\thesubsection.\arabic{theorem}}

\subsection{$K3$ surfaces}  \label{sect:k3}

Let $S$ be a $K3$  surface (by definition, $h^1(\calo_S)=0$ and
$K_S =0$) and $L$ an ample line bundle on $S$. Despite the fact that
these surfaces have been studied extensively and very much is known
about them, remarkably little is known about Seshadri constants on
$K3$ surfaces.

Of course if $L$ is globally generated then $\varepsilon (L,x) \geq
1$ for all $x \in S$ by Proposition \ref{lowspan}.
Non-globally generated ample line bundles on $K3$  surfaces
have been characterized in \cite{S-D}: In this case $L = kE + R$,
where $k \geq 3$, $E$ is a smooth elliptic curve and $R$ a smooth
rational curve such that $E.R=1$. In particular $|E|$ is an elliptic
pencil on $S$ such that $E\cdot L=1$. It follows that $\varepsilon (L,x)
=1$, unless $x$ is a singular point of one of the (finitely many)
singular fibers of $|E|$, in which case $\varepsilon (L,x)
=\frac{1}{2}$ \cite[Prop. 3.1]{B-D-S}. Again this is a case where
the lower bound in Question \ref{eff-low-con} is reached,
and the $K3$ surface is forced to have Picard number $\geq 2$.

Exact values for Seshadri constants in the special case of smooth
quartic surfaces in $\P^3$ have been computed in \cite{Bau97}.

\begin{theorem}[Quartic surfaces]\label{thm:k3q}
    Let $S \subset \P^3$ be a smooth quartic surface.
    Then:

    (a) $\varepsilon (\calo_S(1)) =1$ if and only if $S$ contains a line.

    (b) $\varepsilon (\calo_S(1)) =\frac{4}{3}$ if and only if there is a
point $x \in S$ such that
    the Hesse form vanishes at $x$ and $S$ does not contain any lines.

    (c) $\varepsilon (\calo_S(1)) =2$ otherwise.

    Moreover, the cases (a) and (b) occur on sets of codimension one in the
moduli space of quartic surfaces.
\end{theorem}

(The Hesse form of a smooth surface in $\P^3$ is a quadratic form on
the tangent bundle of $S$, cf. \cite[Sect. 1]{Bau97}.) In particular
$\varepsilon (\calo_S(1)) =2$ on a general quartic surface. Since
the proof very strongly uses the fact that the surface lies in
$\P^3$, it seems very difficult to generalize it to $K3$
surfaces of higher degrees. Nevertheless, a generalization holds in
the case of Picard number one, by the following result \cite[Thm.]{Knu07}:

\begin{theorem}[$K3$ surfaces with $\rho(S)=1$] \label{thm:k3sq}
    Let $S$ be a $K3$ surface with
    $\mbox{Pic} \; S \simeq \Z[L]$ such that $L^2$ is a square. Then
    $\varepsilon (L)= \sqrt{L^2}$.
  \end{theorem}

This result is a corollary of the following more general lower bound
proved in \cite[Corollary]{Knu07}, which can be seen as an extension of
Theorem \ref{thm:steffens} to \emph{all} points on the surface:

\begin{theorem}[Lower bounds on $K3$ surfaces with
$\rho(S)=1$]\label{thm:k3lb}
    Let $S$ be a $K3$ surface with $\mbox{Pic} \; S\simeq \Z[L]$.

    Then either
    \[ \varepsilon (L) \geq \lfloor \sqrt{L^2} \rfloor, \]
    or
\addtocounter{theorem}{1}
    \begin{equation} \label{eq:exc}
       (L^2,\varepsilon(L)) \in \Big\{ (\alpha^2+\alpha-2,
\alpha-\frac{2}{\alpha+1}),
                                    (\alpha^2+ \frac{1}{2}\alpha-
\frac{1}{2},\alpha-\frac{1}{2\alpha+1})
\Big\}
    \end{equation}
    for some $\alpha \in \N$. (Note that in fact $\alpha= \lfloor
\sqrt{L^2} \rfloor$.)
\end{theorem}

In the two exceptional cases (\ref{eq:exc}) of the theorem, the
proof shows that there has to exist a point $x \in S$ and an
irreducible rational curve $C \in |L|$ (resp. $C \in |2L|$) such
that $C$ has an ordinary singular point of multiplicity $\alpha+1$
(resp. $2\alpha+1$) at $x$ and is smooth outside $x$, and
$\varepsilon (L)= L\cdot C/\mult_x C$.

By a well-known result of Chen \cite{ch}, rational curves in the
primitive class of a {\it general} $K3$ surface in the moduli space
are nodal. {\it Hence the first exceptional case in (\ref{eq:exc})
cannot occur on a {\it general} $K3$ surface in the moduli space}
(as $\alpha \geq 2$). If $\alpha=2$, so that $L^2=4$, this special
case is case (b) in Theorem \ref{thm:k3q} above. As one also expects
that rational curves in any multiple of the primitive class on
a~{\it general} $K3$ surface are always nodal (cf. \cite[Conj.
1.2]{ch2}), one may expect that also the second exceptional case in
(\ref{eq:exc}) cannot occur on a {\it general} $K3$ surface.

\setcounter{theorem}{0}
\renewcommand{\thetheorem}{\thesubsection.\arabic{theorem}}
\renewcommand{\theequation}{\thesubsection.\arabic{theorem}}

\subsection{Surfaces of general type}  \label{sect:gen}

Concrete bounds at single points for the canonical divisor have been found
recently, see \cite[Theorem 1]{BauSze08}:

\begin{theorem}[Bounds for the canonical divisor a arbitrary
point]\label{arbitrary}
    Let $S$ be a smooth projective surface such that the canonical
    divisor $K_S$ is big and nef and let $x$ be any point on $S$.
\begin{itemize}
\item[\rm(a)]
    One has $\eps(K_S,x)=0$ if and only if $x$ lies on one of
    finitely many $(-2)$-curves on $X$.
\item[\rm(b)]
    If $0<\eps(K_S,x)<1$, then there is an integer $m\ge 2$ such
    that
    $$
       \eps(K_S,x)=\frac{m-1}m \ ,
    $$
    and there is a Seshadri curve $C\subset S$ such that $\mult_x(C)=m$ and
    $K_S\cdot C=m-1$.
\item[\rm(c)]
    If $0<\eps(K_S,x)<1$ and $K_S^2\ge 2$, then either
    \begin{itemize}
    \item[\rm(i)]
       $\eps(K_S,x)=\frac12$ and $x$ is the double point of
       an irreducible
       curve $C$ with arithmetic genus $p_a(C)=1$ and $K_S\cdot
       C=1$, or
    \item[\rm(ii)]
       $\eps(K_S,x)=\frac 23$ and $x$ is a triple point of
       an irreducible
       curve $C$ with arithmetic genus $p_a(C)=3$ and $K_S\cdot C=2$.
    \end{itemize}
\item[\rm(d)]
    If $0<\eps(K_S,x)<1$ and $K_S^2\ge 3$, then only case {\rm(c)(i)} is
possible.
\end{itemize}
\end{theorem}

    It is well known that the bicanonical system $|2K_S|$ is base
     point free on almost all surfaces of general type. For such
    surfaces one easily gets the lower bound $\eps(K_S,x)\geq 1/2$
    for all $x$ outside the contracted locus.
    However, in general one only knows that $|4K_S|$ is base point
    free, which gives a lower bound of $1/4$. The theorem shows
    in particular that
    one has $\eps(K_S,x)\ge 1/2$ in all cases. Moreover, by Example
\ref{exa:efflb}, the bound is sharp.
    It is not known whether all values $(m-1)/m$ for arbitrary $m\ge
    2$ actually occur.
    As part (c) of Theorem \ref{arbitrary} shows, however,
    values $(m-1)/m$ with $m\ge 4$ can occur
    only in the case $K_S^2=1$. It is shown in \cite[Example 1.3]{BauSze08}
    that curves as in (c)(i) actually exist on surfaces with
    arbitrarily large degree of the canonical bundle.
    In other words, one cannot strengthen the result by imposing
    higher bounds on $K_S^2$. It is not known
    whether curves as in (c)(ii) exist.

   As for values at very general points we have the following bound
   (cf. \cite[Thms. 2 and 3]{BauSze08}).

\begin{theorem}[Positivity of the canonical divisor at very general
points] \label{general}
    Let $S$ be a smooth projective surface such that $K_S$ is big
    and nef.

   If $K_S^2\ge 2$, then $\eps(K_S,1)>1$.

   If $K_S^2\ge 6$, then
       $\eps(K_S,1)\ge 2$ with equality occurring
   if and only if $X$ admits a~genus $2$
       fibration $X\to B$ over a smooth curve $B$.
\end{theorem}

    A somewhat more general statement is given in
\cite[Props. 2.4 and 2.5]{BauSze08}.

\endgroup

\begingroup

\section{S-slope and fibrations by Seshadri curves} 
   As already observed in \ref{sesfun}, the Seshadri constant is
   a lower semi-continuous function of the point. In particular there is a number,
   which we denote by $\eps(X,L;1)$, such that it is the maximal value
   of the Seshadri function. This maximum is attained for a \emph{very general} point $x$.
   Whereas there is no general lower bound on
   values of Seshadri constants at arbitrary points of $X$, the numbers
   $\eps(X,L;1)$ behave much better. It was first observed by Ein and Lazarsfeld \cite{EinLaz93}
   that there is the following universal lower bound on surfaces.
\begin{theorem}[Ein-Lazarsfeld lower bound on surfaces]\mylabel{ellower}
   Let $X$ be a smooth projective surface and $L$ a nef and big line bundle on $X$.
   Then
   $$\eps(X,L;1)\geq 1.$$
\end{theorem}
   It is quite natural to expect that the same bound is valid in arbitrary dimension.
   However up to now the best result in this direction is the following
   result proved by Ein, K\"uchle and Lazarsfeld \cite{EKL}.
\begin{theorem}[Lower bound in arbitrary dimension]\mylabel{ekllower}
   Let $X$ be a smooth projective variety of dimension $n$ and
   $L$ a nef and big line bundle on $X$. Then
   $$\eps(X,L;1)\geq\frac{1}{n}.$$
\end{theorem}
   There has been recently considerable interest in bounds of this type and
   there emerged several interesting improvements in certain special cases.
   Most notably, if $X$ is a threefold, then Nakamaye \cite{Nak04}
   shows $\eps(X,L;1)\geq\frac{1}{2}$ for $L$ an ample line bundle on $X$.
   Under the additional assumption that the anticanonical divisor $-K_X$
   is nef the inequality of Theorem \ref{ekllower} is further improved by Broustet \cite{Bro07}
   who shows that $\eps(X,L;1)\geq 1$ holds in this case.

   The simple example of the projective plane $X$ with $L=\calo_{\P^2}(1)$
   shows that one cannot improve the bound in Theorem \ref{ellower}. One could
   hope however that this bound could be influenced by the degree of $L$. The
   following example shows that this is not the case.
\begin{example}[Polarizations of large degree and low Seshadri constants]\rm
   There exist ample line bundles $L$
   on smooth projective surfaces such that $\eps(X,L;1)=1$,
   with $L^2$ arbitrarily large.

   Consider for instance the product $X=C\times D$ of two smooth irreducible curves,
   and denote by a slight abuse of notation the fibers of both
   projections again by $D$ and $C$.
   The line bundles $L_m=mC+D$ are ample and we have $L_m\cdot C=1$, so that
   in any event $\eps(L_m,x)\le 1$ for every point $x\in X$.
   One has in fact $\eps(L_m,x)=1$, which can be seen as follows:
   If $F$ is any irreducible
   curve different from the fibers of the projections
   with $x\in F$, then we may take a
   fiber $D'$ of the first projection with $x\in D'$, and we have
   $$
      L_m\cdot F \ge D'\cdot F \ge \mult_x(D')\cdot\mult_x(F)
      \ge\mult_x(F)
   $$
   which implies $\eps(L_m,x)\ge 1$.
   So $\eps(L_m,x)=1$, but
   on the other hand $L_m^2=2m$
   is unbounded.
\end{example}

   This kind of behavior is of course not specific for dimension $2$,
   one can easily generalize it to arbitrary dimension.
   Interestingly enough Nakamaye \cite{Nak03} observed that the above
   example is in a sense a unique way to produce low Seshadri constants
   in every point. His result was strengthened and clarified considerably
   in a series of papers \cite{SzeTut04}, \cite{SyzSze07}, \cite{SyzSze08}, \cite{KSS}.
   We summarize below what is known up to now. To this end we introduce first
   the following quantity.
\begin{definition}[S-slope]\mylabel{seshadri slope}\rm
   Let $X$ be a smooth projective variety and $L$ a big and nef line bundle
   on $X$. We define the
   \emph{S-slope} of $L$ as
   $$\sigma(X,L):=\frac{\eps(X,L;1)}{\sqrt[n]{L^n}}\;.$$
\end{definition}
   Note that by Proposition \ref{upperbound} the number in the denominator
   is the upper bound on $\eps(X,L;1)$
   (and hence on $\eps(X,L;x)$ for any $x\in X$).
\begin{definition}[Seshadri fibration]\mylabel{sesfib}\rm
   We say that a surface $X$ is \emph{fibred by Seshadri curves} of $L$ if there
   exists a surjective morphism $f:X\lra B$ onto a complete curve $B$ such that
   for $b\in B$ general the fiber $F_b=f^{-1}(b)$ computes $\eps(X,L;x)$
   for a general $x\in F_b$.\\
   In case of multi-point Seshadri constants we say that $X$
   is fibred by Seshadri curves of $L$ if
   there exists
   a surjective morphism $f:X\lra B$ onto a complete curve $B$ such that
   for $b\in B$ general, the fiber $F_b=f^{-1}(b)$ computes $\eps(X,L;P_1,\dots,P_r)$
   for a general $r$-tuple $P_1,\dots,P_r\in X$ such that
   $\left\{P_1,\dots,P_r\right\}\cap F_b\neq\emptyset$.
\end{definition}
   On surfaces we have
   the following classification.
\begin{theorem}[S-slope on surfaces]\label{single_point}\mylabel{sslopesingle}
   Let $X$ be a smooth surface and $L$ an ample line bundle on $X$. If
   $$\sigma(X,L)<\frac{\sqrt{7}}{\sqrt{8}}\;,$$
   then
   \begin{itemize}
      \item[(a)] either $X$ is fibred by Seshadri curves or
      \item[(b)] $X$ is a smooth cubic surface in $\P^3$ with $L=\calo_X(1)$ and
      $\sigma(X,L)=\frac{\sqrt{3}}{2}$ in this case, or
      \item[(c)] $X$ is a smooth rational surface such that for a general
      point $x\in X$ there is a curve $C_x$ of arithmetic genus $3$ having
      multiplicity $3$ at $x$ and $C_x^2=7$. In this case $\sigma(X,L)=\frac{\sqrt{7}}{3}$.
   \end{itemize}
\end{theorem}
\begin{remark}\rm
   We don't know if surfaces as in Theorem \ref{sslopesingle}(c) exist.
\end{remark}
   The strategy to prove Theorem \ref{sslopesingle} is to consider
   classes of Seshadri curves of $L$ in the Hilbert scheme. In one
   of its components there must be an algebraic family of such curves.
   Then one invokes a bound on the self-intersection of these
   curves in the spirit of \cite{Xu95}. This either leads to the case
   when $C_x^2=0$, hence a multiple of $C_x$ gives a morphism onto
   a curve and we can take the Stein factorization of this morphism, or
   gives restrictions on curves $C_x$ strong enough in order to
   characterize exceptional cases.

   Definition \ref{seshadri slope} generalizes easily to the multi-point case.
\begin{definition}[Multi-point S-slope]\mylabel{multi-s-slope}\rm
   Let $X$ be a smooth projective variety and $L$ a big and nef line bundle
   on $X$. We define the
   \emph{multi-point S-slope} of $L$ as
   $$\sigma(X,L;r):=\frac{\eps(X,L;r)}{\sqrt[n]{L^n/r}}\;.$$
\end{definition}
   The results presented in \cite{SyzSze07} and \cite{SyzSze08} may be summarized
   in the following multi-point counterpart of Theorem \ref{sslopesingle}.
\begin{theorem}\label{sslopemulti}
   Let $X$ be a smooth surface and $L$ an ample line bundle on $X$.
   Let $r\geq 2$ be an integer. If
   $$\sigma(X,L;r)<\sqrt{\frac{2r-1}{2r}}\;,$$
   then
   \begin{itemize}
      \item[(a)] either $X$ is fibred by Seshadri curves or
      \item[(b)] $X$ is a surface of minimal degree in $\P^r$ with $L=\calo_X(1)$ and
      $\sigma(X,L;r)=\sqrt{\frac{r-1}{r}}$ in this case.
   \end{itemize}
\end{theorem}

\endgroup

\begingroup

\section{Algebraic manifestation of Seshadri constants}
   In this section we apply results on Seshadri constants to a problem
   of commutative algebra concerning comparisons of powers of a
   homogeneous ideal in a~polynomial ring with symbolic powers of the
   same ideal.

   To begin, let $R=k[x_0,\ldots,x_N]$ be a polynomial ring in $N+1$
   indeterminates $x_i$ over an algebraically closed field $k$ of
   arbitrary characteristic. We will often regard $R$ as the
   homogeneous coordinate ring $R=k[\P^N]$ of projective $N$-space
   over $k$.

\setcounter{theorem}{0}
\renewcommand{\thetheorem}{\thesubsection.\arabic{theorem}}
\renewcommand{\theequation}{\thesubsection.\arabic{theorem}}

\subsection{Symbolic powers, ordinary powers and the containment problem}

   Let $I\subseteq R$ be a \emph{homogeneous\/} ideal, meaning
   $I=\oplus_i I_i$, where the homogeneous component $I_i$ of $I$ of
   degree $i$ is the $k$-vector space span of all forms $F\in I$ of
   degree $i$.
\begin{definition}[Symbolic power]\mylabel{sympow}\rm
   Given an integer $m\ge1$, the $m$th \emph{symbolic
   power\/} $I^{(m)}$ of $I$ is the ideal
   $$I^{(m)}=\cap_{P\in \Ass(I)}(R\cap I^mR_P)\;.$$
Equivalently,
$$I^{(m)}=R\cap I^mR_U,$$
where $R_U$ is the localization with respect to the set
$U=R-\cup_{P\in\Ass(I)}P$.

\end{definition}
\begin{remark}[Homogeneous primary decomposition]\rm
   All associa\-ted pri\-mes of a~homogeneous ideal are
   themselves homogeneous, and the primary components of a homogeneous
   ideal, meaning the ideals in a primary decomposition, can always be
   taken to be homogeneous (see p. 212, \cite{refBHAbh}). Such a primary decomposition is said to be
   a \emph{homogeneous primary decomposition}; when we refer to
   a~primary decomposition of a homogeneous ideal, we will always mean a
   homogeneous primary decomposition. With this convention, given a
   primary decomposition $I^m=\cap_{P\in{\Ass}(I^m)} Q_P$, where for
   each associated prime $P$ of $I^m$, $Q_P$ denotes the primary
   component of $I^m$ corresponding to $P$, the symbolic power
   $I^{(m)}$ is just $\cap_{P\in S}Q_P$, where $S$ is the set of
   $P\in{\Ass}(I^m)$ such that $Q_P$ is contained in some associated
   prime of $I$.
\end{remark}

\begin{examples}[Some symbolic power examples]\label{sympowersexamples}\rm
Let $I\subseteq R$ be a homogeneous ideal.
By the definition it follows that $I^m\subseteq I^{(m)}$ for all $m\ge 1$,
and by \cite[Proposition 4.9]{BHrefAM} we have $I=I^{(1)}$, but it can happen
that $I^m\subsetneq I^{(m)}$ when $m>1$; see Example \ref{BHfatpts}.
However, by a result of Macaulay, if $I$ is a complete intersection, then $I^{(m)}=I^m$ for all $m\ge1$
(see the proof of Theorem 32 (2), p. 110, \cite{Mat70}).
If $I$ is a radical homogeneous ideal with associated primes
$P_1,\dots,P_j$, then $I=P_1\cap \cdots\cap P_j$ and
$I^{(m)}=P_1^{(m)}\cap \cdots\cap P_j^{(m)}$, where
$P_i^{(m)}$ is the smallest primary ideal containing $P_i^m$.
Thus for an ideal $I=\cap_i P_i$ of a finite set of points $p_1,\dots,p_j\in \P^N$,
where $P_i$ is the ideal generated by all forms vanishing at $p_i$, we have
$I^{(m)}=P_1^m\cap \cdots\cap P_j^m$.
\end{examples}
The problem we wish to address here is that of comparing powers of an ideal $I$
with symbolic powers of $I$. The question of when $I^{(m)}$ contains $I^r$
has an easy complete answer.

\begin{lemma}[Containment condition]\mylabel{concon}
   Let $0\ne I\subsetneq R$ be a homogeneous ideal. Then
   $I^r\subseteq I^{(m)}$ if and only if $r\ge m$.
\end{lemma}
\proof
   Clearly, $r\ge m$ implies $I^r\subseteq I^m \subseteq I^{(m)}$.

   Conversely, say $r<m$ but $I^r\subseteq I^{(m)}$.
   Since  $I^r\subseteq I^{(m)}$, we have $I^{(r)}\subseteq I^{(m)}$,
   and since $r<m$ we have $I^m\subseteq I^r$, so
   $I^{(m)}\subseteq I^{(r)}$ and hence $I^{(r)}=I^{(m)}$. Thus there
   is an associated prime $P$ of $I$ such that $I^rR_P=I^mR_P\neq(1)$
   and so $I^rR_P=I^mR_P=(I^rR_P)(I^sR_P)$, where $s+r=m$. By
   Nakayama's lemma, this implies $I^rR_P=0$, contradicting $0\ne I$.
\endproof
\noindent The question, on the other hand, of when $I^r$ contains $I^{(m)}$
turns out to be very delicate. This is the main problem we will consider here.
\begin{problem}[Open Problem]\rm
Let $I\subseteq R$ be a homogeneous ideal.
Determine for which $r$ and $m$ we have $I^{(m)}\subseteq I^r$
\end{problem}
In order to make a connection of this problem to computing Seshadri constants
we will need the following definition.
Let $M=(x_0,\ldots,x_N)$ be the maximal homogeneous ideal of $R$.

\begin{definition}[$M$-adic order of an ideal]\label{madord}\rm
   Given a homogeneous ideal $0\ne I\subseteq R$, let
   $\alpha(I)$ be the $M$-adic order of $I$; i.e., the least $t$ such that $I$ contains a~nonzero
   homogeneous element of degree $t$; equivalently, $\alpha(I)$ is the
   least $t$ such that $I_t\ne 0$.
\end{definition}
   For any homogeneous ideal $0\ne I\subseteq R$, it is easy to see
   that $\alpha(I^m)=m\alpha(I)$, but for symbolic powers we have just
   $\alpha(I^{(m)})\le m\alpha(I)$; as Example \ref{BHfatpts}
   shows, this inequality can be strict. First a definition.
\begin{definition}[Fat point subscheme]\mylabel{fatptsub}\rm
   Given distinct points $p_1,\ldots,p_j\in \P^N$, let $I(p_i)$ be
   the maximal ideal of the point $p_i$. Given a
   $0$-cycle $Z=m_1p_1+\cdots+m_jp_j$ with positive integers $m_i$, let
   $I(Z)$ denote the ideal $\cap_i I(p_i)^{m_i}$. We also write
   $Z=m_1p_1+\cdots+m_jp_j$ to denote the subscheme defined by $I(Z)$.
   Such a~subscheme is called a \emph{fat point\/} subscheme.
\end{definition}
   Now we consider an easy example of a fat point subscheme of $\P^2$.
\begin{example}[The power and symbolic power can differ]\mylabel{BHfatpts}\rm
   Given
   $Z=p_1+\cdots+p_j$ and $m\ge1$, $mZ$ is the subscheme
   $mp_1+\cdots+mp_j$, and we have $I(mZ)=I(Z)^{(m)}$. The ideal
   $I(mZ)$ is generated by all forms that vanish to order at least $m$
   at each point $p_i$. If $N=2$ and $I=I(p_1+p_2+p_3)$, where
   $p_1=(1:0:0)$, $p_2=(0:1:0)$ and $p_3=(0:0:1)$, then $\alpha(I)=2$
   so $\alpha(I^2)=4$ but, since $x_0x_1x_2\in I^{(2)}$, we have
   $\alpha(I^{(2)})\le3$ (and in fact this is an equality), and thus
   $\alpha(I^{(2)})<2\alpha(I)=4$, hence
   $I^2\subsetneq I^{(2)}$.
\end{example}

\setcounter{theorem}{0}
\renewcommand{\thetheorem}{\thesubsection.\arabic{theorem}}
\renewcommand{\theequation}{\thesubsection.\arabic{theorem}}

\subsection{Measurement of growth and Seshadri constants}

   An interesting problem, pursued in \cite{BHrefAV} and
   \cite{BHrefCHHT}, is to determine how much bigger $I^{(m)}$ is than
   $I^m$. Whereas \cite{BHrefCHHT} uses local cohomology to obtain an
   asymptotic measure of $I^{(m)}/I^m$, \cite{BHrefAV} uses the
   regularity of $I$ to estimate how big $I^{(m)}/I^m$ is. An
   alternative approach is to use an asymptotic version of
   $\alpha$ \cite{BHrefBH}.
\begin{definition}[Asymptotic $M$-adic order]\mylabel{asmadord}\rm
   Given a
   homogeneous ideal $0\ne I\subseteq R$, then $\alpha(I^{(m)})$ is
   defined for all $m\ge 1$ and we define
   $$\gamma(I)=\lim_{m\to\infty}\frac{\alpha(I^{(m)})}{m}\;.
   $$
\end{definition}
   Because $\alpha$ is subadditive (i.e.,
$\alpha(I^{(m_1+m_2)})\le \alpha(I^{(m_1)})+\alpha(I^{(m_2)})$), this limit
   exists (see \cite[Lemma 2.3.1]{BHrefBH}, or
   \cite[Remark III.7]{BHrefHR}).
\begin{lemma}[Positivity of $\gamma$]
   Given a
   homogeneous ideal $0\ne I\subsetneq R$, then
   $\gamma(I)\ge1$.
\end{lemma}
\proof
   To see this, consider $M=(x_0,\ldots,x_N)$. Let $P\in\Ass(I)$.
Then $I^{(m)}\subseteq P^{(m)}$. But $P$ is homogeneous, so $P\subseteq M$, hence $P^{(m)}\subseteq M^{m}$
by Corollary 1 of \cite{refBHEH}. Thus $m=\alpha(M^m)\le
   \alpha(I^{(m)})$, hence $1\le \gamma(I)$.
\endproof

\begin{remark}[$\gamma$, the containment problem and Seshadri constants]\label{BHscrem}\rm We note that $0\ne I\subsetneq R$
   guarantees that $\alpha(I)$ is defined, and that $\gamma(I)$ is
   defined and nonzero. The quantity $\gamma(I)$ is useful not only for
   studying when $I^m\subsetneq I^{(m)}$ but, as we will see in Lemma \ref{BHlowerboundlemma},
   also for studying when
   $I^{(m)}\subseteq I^r$. We also will relate $\gamma(I)$ to Seshadri constants.

   First we see how $\alpha(I)/\gamma(I)$ gives an asymptotic indication of
   when $I^m\subsetneq I^{(m)}$ in case
   $0\ne I\subsetneq R$. Note by subadditivity we have for all $m$ that
   $$\gamma(I)=\lim_{t\to\infty}\frac{\alpha(I^{(tm)})}{tm}\leq
     \frac{\alpha(I^{(m)t})}{mt}= \frac{\alpha(I^{(m)})}{m}\leq\frac{\alpha(I^{m})}{m}=\alpha(I)\;.$$
   Thus, for example, $\alpha(I)/\gamma(I)>1$ if and only if
   $\alpha(I^m)>\alpha(I^{(m)})$ for some (equivalently, infinitely
   many) $m>1$, and hence $\alpha(I)/\gamma(I)>1$ implies
   $I^m\subsetneq I^{(m)}$ for some (equivalently, infinitely many)
   $m>1$.

   As pointed out in \cite{BHrefBH}, $\gamma(I)$ is in some cases
   related to a suitable Seshadri constant. In particular, if
   $Z=p_1+\cdots+p_j\subset {\bf P}^N$, then one defines the Seshadri
   constant (cf. Definition \ref{defhighdim})
   $$\eps(N,Z):=\eps_{N-1}(\P^N,\calo(1);Z)=\root{N-1}\of {\inf
     \left\{\frac{{\deg(H)}}{{\Sigma_{i=1}^j
     \mult_{p_i}H}}\right\}}\;,
   $$
   where the infimum is taken with
   respect to all hypersurfaces $H$ through at least one of the points
   $p_i$. It is clear from the definitions that
   $$\gamma(I(Z))\ge j\cdot (\varepsilon(N,Z))^{N-1}\;.
   $$
   If the points $p_i$ are generic, then
   equality holds (see \cite[Lemma 2.3.1]{BHrefBH}, or
   \cite[Remark III.7]{BHrefHR}; the idea of the proof is to use the fact that the
   points are generic to show that one can assume that $H$ has the same
   multiplicity at each point $p_i$).
\end{remark}

\setcounter{theorem}{0}
\renewcommand{\thetheorem}{\thesubsection.\arabic{theorem}}
\renewcommand{\theequation}{\thesubsection.\arabic{theorem}}

\subsection{Background for the containment problem}

   It is not so easy to determine for which $r$ we have
   $I^{(m)}\subseteq I^r$. It is this problem that is the motivation
   for \cite{BHrefBH}, which develops an asymptotic approach to this
   problem. If $I$ is nontrivial (i.e., $0\ne I\subsetneq R$), then
   the set $\{m/r: I^{(m)}\not\subseteq I^r\}$ is nonempty, and we define
   $\rho(I)=\sup\{m/r: I^{(m)}\not\subseteq I^r\}$; \emph{a priori\/}
   $\rho(I)$ can be infinite. When an upper bound does exist, we see
   that $I^{(m)}\subseteq I^r$ whenever $m/r>\rho(I)$.

   Swanson \cite{BHrefS} showed that an upper bound exists for many
   ideals $I$. This was the inspiration for the papers \cite{BHrefELS}
   and \cite{BHrefHH}, whose results imply that $\rho(I)\le N$ for any
   nontrivial homogeneous ideal $I\subset \P^N$. In fact, if we define
   $\codim(I)$ to be the maximum height among associated primes
   of $I$ other than $M$, then it follows from \cite{BHrefELS} and
   \cite{BHrefHH} that $\rho(I)\le \min\{N, \codim(I)\}$. (If $M$
   is an associated prime of $I$, as happens if $I$ is not saturated, then
   $I^m=I^{(m)}$ for all $m\ge 1$, since every homogeneous primary
   ideal in $R$ is contained in $M$.)

   This raises the question of whether the bound $\rho(I)\le N$ can be
   improved. Results of \cite{BHrefBH} show that this bound and the
   bounds $\rho(I)\le \hbox{codim}(I)$ are optimal, in the sense that
   $\sup\{\rho(I): 0\ne I\subsetneq R\hbox{ homogeneous}\}=N$, and,
   when $e\le N$,
\addtocounter{theorem}{1}
   \begin{equation}\label{rhoopt}
   \sup\{\rho(I): 0\ne I\subsetneq R\hbox{ homogeneous of }\hbox{codim}(I)=e\}=e\;.
   \end{equation}
   To justify this we introduce the following arrangements
   of linear subspaces.
\begin{notation}[Generic arrangements of linear subspaces]\rm
   Let $H_1,\ldots,H_s$ be $s > N$ generic hyperplanes in $\P^N$.
   Let $1\le e\le N$ and let $S\subset\{1,2,\ldots,s\}$ with $|S|=e$.
   We define the scheme
   $Z_S(N,s,e)$ to be $\cap_{i\in S}H_i$, so $Z_S(N,s,e)$ is a~linear
   subspace of ${\bf P}^N$ of codimension $e$.
We also let $Z(N,s,e)$ be the
   union of all $Z_S(N,s,e)$ with $|S|=e$.
\end{notation}
   The following result,  \cite[Lemma 2.3.2(b)]{BHrefBH}, as applied
   in Example \ref{BHstarconfigs}, justifies \eqnref{rhoopt}:

\begin{lemma}[Asymptotic noncontainment]\label{BHlowerboundlemma}
   Given a
   homogeneous ideal $0\ne I\subsetneq R$ and $m/r<\frac{\alpha(I)}{\gamma(I)}$, then
   $I^{(mt)}\not\subseteq I^{rt}$ for all $t\gg0$; in particular,
   $\frac{\alpha(I)}{\gamma(I)}\le\rho(I)$.
\end{lemma}

\begin{example}[Sharp examples of \cite{BHrefBH}]\label{BHstarconfigs}\label{BHexamples}\rm
   We write $I(mZ(N,s,e))$ to denote
   $I(Z(N,s,e))^{(m)}$. Then $\alpha(I(mZ(N,s,e)))=ms/e$ if $e|m$ and
   by Lemma \ref{BHnumlem} $\alpha(I(Z(N,s,e)))=s-e+1$. Thus
   $\gamma(I(Z(N,s,e)))=s/e$ and $\rho(I(Z(N,s,e)))\ge e(s-e+1)/s$.
   Keeping in mind $\rho(I(Z(N,s,e)))\le e$ (which holds by \cite{BHrefHH} since
$\hbox{codim}(I(Z(N,s,e)))=e$), we now see $\lim_{s\to
   \infty}\rho(I(Z(N,s,e)))=e$, so the bounds of \cite{BHrefELS} and
   \cite{BHrefHH} are sharp.
\end{example}

\begin{remark}[A Seshadri constant computation]\rm
   Let $I=I(Z(N,s,N))$. It is interesting to note, by
   \cite[Theorem 2.4.3(a)]{BHrefBH}, that $\rho(I)=\alpha(I)/\gamma(I)$, and hence
   $\rho(I)=N(s-N+1)/s$. By an argument similar to that of
   \cite[Lemma 2.3.1]{BHrefBH}, discussed above in Remark \ref{BHscrem},
   we can express $\rho(I)$ in terms of the Seshadri constant $\eps(N,Z)$.
In particular, $\gamma(I)=|Z|\cdot\eps(N,Z)^{N-1}$ holds, and thus we obtain
   $$\eps(N,Z)=\root{N-1}\of{\frac{s}{N\binom{s}{N}}}\;.$$
\end{remark}

\setcounter{theorem}{0}
\renewcommand{\thetheorem}{\thesubsection.\arabic{theorem}}
\renewcommand{\theequation}{\thesubsection.\arabic{theorem}}

\subsection{Conjectural improvements}

   Even though the bound $\rho(I)\le N$ is optimal (in the sense that
   for no value $d$ smaller than $N$ will $\rho(I)\le d$ hold for all
   nontrivial homogeneous ideals $I$), we can try to do better. The
   bound $\rho(I)\le \hbox{codim}(I)$ can be rephrased as saying $I^{(m)}\subseteq
   I^r$ if $m>r\,\hbox{codim}(I)$. In fact the results of \cite{BHrefELS} and
   \cite{BHrefHH} imply the slightly stronger result that
   $I^{(m)}\subseteq I^r$ if $m\ge r\,\hbox{codim}(I)$. As a next step, we can ask for
   the largest {\it integer\/} $d_e$ such that $I^{(m)}\subseteq I^r$
   whenever $m \ge re-d_{e}$, where $e=\hbox{codim}(I)$.

Examples of Takagi and Yoshida \cite{BHrefTY} support the
   possibility that $I^{(m)}\subseteq I^r$
   holds for $m\ge Nr-1$ (i.e., perhaps it is true that $d_e\ge 1$).
   On the other hand, the obvious fact that $\alpha(I^{(m)})<\alpha(I^r)$ implies
  $I^{(m)}\not\subseteq I^r$ (see Theorem \ref{BHcontainmentcrits}(a) for a reference),
  applied with $m=re-e$ for $e>1$ and $s\gg0$ to $I(mZ(N,s,e))$
of  Example \ref{BHexamples}, shows that $d_e<e$.

For example, the fact that $I^{(2)}$ is not always
   contained in $I^2$, as we saw in Example \ref{BHfatpts}, shows that $d_2<2$
   (at least for $I\subseteq R=k[\P^2]$), and hence
   either $d_2=0$ or $d_2=1$. Proving $d_2=1$ for $R=k[\P^2]$ would provide an
   affirmative answer to an as-of-now still open unpublished question
   raised by Craig Huneke:

\begin{question}[Huneke]\label{BHCraigques}
   Let $I=I(Z)$ where $Z=p_1+\cdots+p_j$ for distinct points $p_i\in
   \P^2$. Then we know $I^{(4)}\subseteq I^2$, but is it also true
   that $I^{(3)}\subseteq I^2$?
\end{question}
The following conjectures are motivated by Huneke's question,
by the fact that $d_e<e$ as we saw above, and by
a number of suggestive supporting examples which we will recall below:

\begin{conjecture}[Harbourne]\label{BHPNconj}
   Let $I$ be a homogeneous ideal with $0\ne I\subsetneq R=k[\P^N]$.
   Then $I^{(m)}\subseteq I^r$ if $m\ge Nr-(N-1)$.
\end{conjecture}
Since $Nr-(N-1)\geq er-(e-1)$ for any positive integers $e\leq N$,
the previous conjecture is a consequence of the following more precise
version of the conjecture:

\begin{conjecture}[Harbourne]\label{BHPNcodimconj}
    Let $I$ be a homogeneous ideal with
    $0\ne I\subsetneq R=k[\P^N]$ and $\codim(I)=e$.
    Then $I^{(m)}\subseteq I^r$ if $m\ge er-(e-1)$.
\end{conjecture}
We conclude by recalling evidence in support of these conjectures.

\begin{example}[Examples of Huneke]\label{CraigProof}\rm
After receiving communication of these conjectures,
Huneke re-examined the methods of
\cite{BHrefHH} and noticed that they
implied that Question \ref{BHCraigques} has an affirmative answer in
characteristic 2. More generally, Conjecture \ref{BHPNconj}
is true if $r=p^t$ for $t>0$, where $p={\rm char}(k)>0$ and $I$ is the radical
ideal defining a set of points $p_1,\dots,p_j\in \P^N$.
Huneke's argument uses the fact that in characteristic $p$ taking a Frobenius power $J^{[r]}$
of an ideal $J$ (defined as the ideal $J^{[r]}$ generated by the $r$th powers of elements of $J$)
commutes with intersection. (To see this, note that $J^{[r]}=J\otimes_RS$, where
$\phi^t:R\to R=S$ is the $t$th power of the Frobenius homomorphism.
Tensoring $0\to J_1\cap J_2\to J_1\oplus J_2\to J_1+J_2\to 0$ by $S$ over $R$
gives a short exact sequence. This is because of flatness of Frobenius; see, for example,
 \cite[Lemma 13.1.3, p. 247]{BHrefHS} and \cite{BHrefKunz}.
Comparing the resulting short exact sequence with
$0\to J_1^{[r]}\cap J_2^{[r]}\to J_1^{[r]}\oplus J_2^{[r]}\to J_1^{[r]}+J_2^{[r]}\to 0$ gives
the result.) It also uses the observation that
a large enough power of any ideal $J$ is contained in a given Frobenius power of $J$.
More precisely, if $J$ is generated by $h$ elements, then
$J^m\subseteq J^{[r]}$ as long as $m\ge rh-h+1$.
This is because $J^m$ is generated by monomials
in the $h$ generators, but in every monomial involving a product of at least
$rh-h+1$ of the generators there occurs a factor consisting of
one of the generators raised to the power $r$.

\eatit{(We explain in more detail the commutativity of Frobenius powers
with intersections. Given an ideal $J\subseteq R$,
where ${\rm char}(k)=p>0$, define the Frobenius power $J^{[p^t]}$ to be the ideal generated by
$\phi^t(J)$, where $\phi:R\to R=S$ is the Frobenius map, defined by
$\phi(a)=a^p$. Note $J^{[p^t]}\cong J\otimes_{\phi^t}S$ in $S$ since $\phi$ is flat.
See \cite[Lemma 13.1.3, p. 247]{BHrefHS}, but note that the proof applies to $R$
even though $R$ is not local in this case; see also \cite{BHrefKunz}.
Given ideals $J_1,J_2\subseteq R$, we clearly have an exact sequence
$0\to J_1\cap J_2\to J_1\oplus J_2\to J_1+J_2\to 0$ and hence
$0\to J_1^{[p^t]}\cap J_2^{[p^t]}\to J_1^{[p^t]}\oplus J_2^{[p^t]}\to J_1^{[p^t]}+J_2^{[p^t]}\to 0$.
Tensoring the former by $S$ over $R$ gives the exact sequence
$0\to (J_1\cap J_2)^{[p^t]}\to J_1^{[p^t]}\oplus J_2^{[p^t]}\to J_1^{[p^t]}+J_2^{[p^t]}\to 0$,
and hence $(J_1\cap J_2)^{[p^t]}=J_1^{[p^t]}\cap J_2^{[p^t]}$.)}

In particular, since ideals of points in $\P^N$ are generated by $N$ elements,
following the notation of Example \ref{sympowersexamples} (and keeping in mind that
$r$ must be a power of $p$ here) we have
$$I^{(rN-N+1)}=\cap_i P_i^{rN-N+1}\subseteq
\cap_i P_i^{[r]} = (\cap_i P_i)^{[r]} = I^{[r]}\subseteq I^r.$$
Huneke's argument also applies more generally to show that
Conjecture \ref{BHPNcodimconj} is true for any radical ideal $I$
when $r$ is a power of the characteristic, using the fact that
Frobenius powers commute with colons (see \cite[Proof of part (6) of Theorem 13.1.2, p. 247]{BHrefHS})
and using the fact that $PR_P$ is generated by $h$ elements, where $h$ is the height of the prime $P$.
\end{example}

\begin{example}[Monomial ideals]\label{MonomialIdealProof}\rm
As another example, we now show that Conjecture \ref{BHPNcodimconj} holds for
any monomial ideal $I\subset R$ in any characteristic. We sketch the proof, leaving
basic facts about monomial ideals as exercises.

Consider a monomial ideal $I$; let $P_1,\dots,P_s$ be the
associated primes. These primes are necessarily
monomial ideals and hence are generated by subsets of the variables.
Moreover, $I$ has a primary decomposition $I=\cap_{ij}Q_{ij}$
where the $P_i$-primary component of $I$
is $\cap_jQ_{ij}$ and each $Q_{ij}$ is generated by positive powers of
the variables which generate $P_i$. Let $e$ be the maximum of the heights of $P_i$
and let $m\ge er-r+1$. By definition we then have
$I^{(m)}= \cap_i(I^mR_{P_i}\cap R)$, but
$\cap_i(I^mR_{P_i}\cap R)\subseteq \cap_i((\cap_jQ_{ij})^mR_{P_i}\cap R)$
since
$$I^mR_{P_i}=(\cap_{\{t:P_t\subseteq P_i\}}(\cap_jQ_{tj}))^mR_{P_i}.$$
Clearly, $(\cap_jQ_{ij})^mR_{P_i}\subseteq \cap_jQ_{ij}^mR_{P_i}$
but $Q_{ij}^m$ is $P_i$-primary (hence $Q_{ij}^mR_{P_i}\cap R=Q_{ij}^m$),
so we have
$$\cap_i((\cap_jQ_{ij})^mR_{P_i}\cap R)\subseteq\cap_{ij}Q_{ij}^m.$$
Now, each $Q_{ij}$ is generated by at most $e$ elements, so
$Q_{ij}^m\subseteq Q_{ij}^{[r]}$, where
$J^{[r]}$ is defined for any monomial ideal $J$ to be the ideal generated
by the $r$th powers of the monomials contained in $J$;
thus $\cap_{ij}Q_{ij}^m\subseteq \cap_{ij}\,Q_{ij}^{[r]}$.
Finally, we note that if $J_1$ and $J_2$ are monomial ideals, then
$(J_1\cap J_2)^{[r]}=J_1^{[r]}\cap J_2^{[r]}$ (since $(J_1\cap J_2)^{[r]}$
is generated by the $r$th powers of the least common multiples of the generators of
$J_1$ and $J_2$, while $J_1^{[r]}\cap J_2^{[r]}$ is generated by the least common
multiples of the $r$th powers of the generators of
$J_1$ and $J_2$, and taking $r$th powers commutes with taking
least common multiples).
So we have $\cap_{ij}\,Q_{ij}^{[r]}=(\cap_{ij}Q_{ij})^{[r]}= I^{[r]}\subseteq I^r$,
and we conclude that $I^{(m)}\subseteq I^r$.

\end{example}
The schemes $Z(N,s,N)\subset \P^N$ give additional examples for which Conjecture \ref{BHPNconj} is true.
   These schemes are of particular interest, since, as we saw above, they are
   asymptotically extremal for $\rho$, and thus one might expect
  if the conjecture were false that one of these schemes would provide a
   counterexample.
   In order to see why Conjecture \ref{BHPNconj} is true for symbolic powers of $I(Z(N,s,N))$, we need the
   following theorem, for which we recall the
   notion of the \emph{regularity\/} of an ideal. We need it
   only in a special case:
\begin{quote}
   If $I$ defines a 0-dimensional subscheme of
   projective space, the regularity $\reg(I)$ of $I$ equals the
   least $t$ such that $(R/I)_t$ and $(R/I)_{t-1}$ have the same
   $k$-vector space dimension.
\end{quote}
   As an example, if $I=I(p_1+\cdots+p_j)$
   for distinct generic points $p_i$, then (since the points impose
   independent conditions on forms of degree $i$ as long as $\dim R_i\ge j$) we have
   $\reg(I)=t+1$, where $t$ is the least
   degree such that $\dim(R_t) \ge j$.

\begin{theorem}[Noncontainment and Containment Criteria]\label{BHcontainmentcrits}
   Let $0\ne I \subseteq R=k[\P^N]$ be a homogeneous ideal.
   \begin{itemize}
      \item[(a)] Non-containment Criterion: If $\alpha(I^{(m)})< \alpha(I^r)$, then $I^{(m)}\not\subseteq I^r$.
      \item[(b)] Containment Criterion: Assume $\codim(I)=N$.
         If $r\reg(I)\le \alpha(I^{(m)})$, then $I^{(m)}\subseteq I^r$.
   \end{itemize}
\end{theorem}

\begin{proof}
   (a) This is  \cite[Lemma 2.3.2(a)]{BHrefBH}.

   (b) This is  \cite[Lemma 2.3.4]{BHrefBH}.
\end{proof}

In order to verify
   that Question \ref{BHCraigques} has an affirmative answer
   for the case $I=I(Z(2,s,2))$, and that
   $I^{(m)}\subseteq I^r$ holds whenever $m\ge Nr-N+1$ for $I=I(Z(N,s,N))$, we will apply
   Theorem \ref{BHcontainmentcrits}. To do so, we need the following
   numerical results.

\begin{lemma}[Some numerics]\label{BHnumlem} Let $I= I(Z(N,s,e))\subset R=k[\P^N]$.
   \begin{itemize}
      \item[(a)] Then $\alpha(I)=s-e+1$; if $e=N$, then $\alpha(I)=\reg(I)=s-N+1$.
      \item[(b)] If $e|m$, then $\alpha(I^{(m)})=ms/e$.
      \item[(c)] Let $m=iN+j$, where $i\ge 0$ and $0<j\le N$, and let $I=I(Z(N,s,N))$
   where $s>N\ge 1$. Then $\alpha(I^{(m)})=(i+1)s-N+j$.
   \end{itemize}
\end{lemma}

\begin{proof}
   (a) This holds by  \cite[Lemma 2.4.2]{BHrefBH}.

   (b) This holds by  \cite[Lemma 2.4.1]{BHrefBH}.

   (c) See Lemma 8.4.5 and Proposition 8.5.3 of version 1 of
   ArKiv0810.0728 for detailed proofs. (We note that the case $N=2$
   follows easily by using induction and B\'ezout's theorem.)
\end{proof}

\begin{example}[Additional supporting evidence]\label{BHPNstarconfig}\rm
Let $I=I(Z(N,s,N))$.
By Lemma \ref{BHnumlem}(a), $\reg(I)=s-N+1$ and
by Lemma \ref{BHnumlem}(c),
   $\alpha(I^{(m)})=(i+1)s-N+j$, where $m=iN+j$, $i\ge 0$ and $0<j\le
   N$. Thus, if $m=rN-(N-1)=(r-1)N+1$, then $\alpha(I^{(m)})=rs-N+1\geq
r(s-N+1)=r\reg(I)$, and hence
   $I^{(Nr-(N-1))}\subseteq I^r$ by Theorem
   \ref{BHcontainmentcrits}(b).
   Thus Conjecture \ref{BHPNconj} holds for $I=I(Z(N,s,N))$.
   Moreover, when $r=N=2$, we have $I^{(3)}\subseteq I^2$, which shows
   that Question \ref{BHCraigques} has an affirmative answer
   for $I=I(Z(2,s,2))$.
\end{example}

\noindent In our remaining two examples, concerning generic points in projective space,
Seshadri constants play a key role.

\begin{example}[Generic points in $\P^2$]\label{BHgenP2ex}\rm
By  \cite[Theorem 4.1]{BHrefBH}, Huneke's question again has an
affirmative answer if $I$ is the ideal of generic points $p_1,\dots,p_j\in\P^2$.
More generally, by \cite[Remark 4.3]{BHrefBH} we have $\rho(I)<3/2$ when
$I$ is the ideal of a finite set of generic points in $\P^2$. It follows that
$I^{(m)}\subseteq I^r$ whenever $m/r\ge 3/2$. Since $m\ge 2r-1$ implies
that either $m/r\ge 3/2$ or $r=1$, Conjecture \ref{BHPNconj}
is true in the case $N=2$ and $I$ is the ideal of generic points in the plane.

The proof that $\rho(I)<3/2$ depends on using estimates for multipoint Seshadri
constants to handle the case that $j$ is large. The few remaining cases
of small $j$ are then handled {\it ad hoc}. We now describe this argument
for large $j$ in more detail.
   Let $I=I(Z)$, where $Z=p_1+\cdots+p_j$ for distinct generic points
   $p_i\in \P^2$. By  \cite[Corollary 2.3.5]{BHrefBH} we have
   $\rho(I)\le \reg(I)/\gamma(I)$. If $j\gg0$, we wish to show
   that $I^{(m)}\subseteq I^r$ for all $m\ge 2r-1$. The proof depends
   on estimating $\eps(2,Z)$ and $\reg(I)$, and then using
   $\gamma(I)=j\cdot\eps(2,Z)$ from Remark \ref{BHscrem} and $\rho(I)\le
   \reg(I)/\gamma(I)$.

   To estimate $\reg(I)$, given that
   $\reg(I)=t+1$ where $t$ is the least degree such that
   $\dim(R_t) \ge j$, use the fact that
   $\dim(R_t)=\binom{t+2}{2}$. It is now not hard to check that
   $\reg(I)\le \sqrt{2j+(1/4)}+(1/2)$ for $j\gg0$. We also have
   $j\cdot\eps(2,Z)\ge \sqrt{j-1}$ for $j\ge 10$ (for characteristic
   0, see \cite{BHrefX}; see the proof of  \cite[Theorem 4.2]{BHrefBH} in
   general).

   Thus for $j\gg0$ we have
   $$\rho(I)\le\reg(I)/\gamma(I) \le (\sqrt{2j+(1/4)}+(1/2))/\sqrt{j-1}\;;
   $$
   for large $j$ this is close to $\sqrt{2}$ and thus less than $3/2$. But
   $m\ge 2r-1$ implies $m/r\ge 3/2>\rho(I)$ for all $r>1$. Thus
   $I^{(m)}\subseteq I^r$ for $r>1$. If $r=1$, then we also have
   $I^{(m)}\subseteq I^{(1)}=I=I^r$.
\end{example}
   Finally, we show that $I^{(Nr-(N-1))}\subseteq I^r$ holds for
   $j\gg0$ if $I=I(Z)$, where $Z=p_1+\cdots+p_j$ for distinct generic
   points $p_i\in \P^N$. The argument is modeled on that used in
   Example \ref{BHgenP2ex}.

\begin{example}[Generic points in $\P^N$]\label{BHgenptsinPN}\rm
Let $I=I(Z)$, where $Z=p_1+\cdots+p_j$ for distinct generic points
$p_i\in {\bf P}^N$. To show $I^{(Nr-(N-1))}\subseteq I^r$, since the
case $r=1$ is clear, it is enough to consider $r\ge2$. As in Example
\ref{BHgenP2ex} $\rho(I)\le
   \reg(I)/\gamma(I)$, so it suffices to show
$\reg(I)/(j(\varepsilon(N,Z))^{N-1})<(rN-(N-1))/r$ for
$j\gg0$, and since $(rN-(N-1))/r$ is least for $r=2$, it is enough
to verify this for $r=2$. To estimate $\reg(I)$, use the facts
that $\hbox{dim}(R_t)=\binom{t+N}{N}$ and $\reg(I)=t+1$, where
$t$ is the least nonnegative integer such that
$j\le \binom{t+N}{N}$. Since $j=(\root{N}\of{N!j})^N/(N!)\le
(x+N)\cdots(x+1)/(N!)$ for $x=\root{N}\of{N!j}-1$,
we see $t\leq \lceil \root{N}\of{N!j}-1 \rceil\leq \root{N}\of{N!j}$
and hence
$\reg(I)\le \root{N}\of{N!j}+1$. Next, for $j\gg0$, we have
$$\frac{j-1}{j}\frac{1}{\root{N}\of{j-1}}=\frac{\root{N}\of{(j-1)^{N-1}}}{j}\leq\eps_{1}(\P^N,\calo(1);Z)$$
by Theorem 1.1 \cite{BHrefK} and
$$\eps_{1}(\P^N,\calo(1);Z)\le \varepsilon(N,Z)$$
by Proposition \ref{SCrltn} (although Proposition \ref{SCrltn}
is stated only for the case $j=1$ of a single point, the proof (see Proposition 5.1.9 \cite{PAG}) carries over for any
$j$). Thus
$$\frac{j-1}{j}\frac{1}{\root{N}\of{j-1}}\le \varepsilon(N,Z)$$
and hence
$$\Big(\frac{j-1}{j}\Big)^{N-2}\root{N}\of{j-1}=
\Big(\frac{j-1}{j}\Big)^{N-1}\frac{j}{{\root{N}\of{j-1}}^{N-1}}\le
j(\varepsilon(N,Z))^{N-1}$$ so
$$\frac{\reg(I)}{j(\varepsilon(N,Z))^{N-1}}\le \frac{\root{N}\of{N!j}+1}{\Big(\frac{j-1}{j}\Big)^{N-2}\root{N}\of{j-1}}$$
for $j\gg0$. But
$$\lim_{j\to\infty}\frac{\root{N}\of{N!j}+1}{\Big(\frac{j-1}{j}\Big)^{N-2}\root{N}\of{j-1}}=\root{N}\of{N!}$$
so
$$\frac{\reg(I)}{j(\varepsilon(N,Z))^{N-1}}<\frac{N+1}{2}$$
follows for $j\gg0$ if we have
$$\root{N}\of{N!}<\frac{N+1}{2},$$
and this is equivalent to $2^NN!<(N+1)^N$. This last is true for
$N=2$, and if it is true for some $N\ge2$, then it holds for $N+1$
(and hence for all $N\ge2$ by induction) if
$$2(N+1)\le (N+2)((N+2)/(N+1))^N,$$
since then
$$2^{N+1}(N+1)!=2(N+1)2^NN!<(N+2)((N+2)/(N+1))^N(N+1)^{N}=(N+2)^{N+1}.$$
But taking $n=N+1$, we can rewrite $2(N+1)\le (N+2)((N+2)/(N+1))^N$
as $2\le
((N+2)/(N+1))^{N+1}=(1+\frac{1}{n})^n=1^n+\binom{n}{1}\frac{1}{n}+\binom{n}{2}\frac{1}{n^2}+\cdots$,
which is obvious.
\end{example}

\endgroup



\bibliographystyle{amsalpha}

\end{document}